\documentclass[12pt,reqno]{amsart}
\usepackage[a4paper,margin=35mm]{geometry}
\usepackage{amssymb}
\usepackage{amsmath}
\usepackage{amsfonts}
\usepackage{amsthm}
\usepackage{amsrefs}
\usepackage{pstricks}
\usepackage{graphicx}
\usepackage{enumerate}        

\usepackage[colorlinks]{hyperref}
\usepackage[english]{babel}
\usepackage[utf8]{inputenc}
\usepackage[T1]{fontenc}

\newcommand{\ga}{\ensuremath{\gamma}}

\newcommand{\lam}{\ensuremath{\lambda}}

\newcommand{\Ga}{\Gamma}

\newcommand{\m}[1]{{ \bf {#1} }}
\newcommand{\vty}[1]{\ensuremath{\mathcal{#1}}}		

\newcommand{\cng}{{\rm Con}}	

\newcommand{\ra}{\to}

\newcommand{\lan}{\langle}
\newcommand{\ran}{\rangle}
\newcommand{\jn}{\vee}
\newcommand{\mt}{\wedge}
\newcommand{\rd}{\slash}
\newcommand{\ld}{\backslash}
\newcommand{\ldm}{\backslash_{_{_{-}}}}
\newcommand{\rdm}{\slash\!\!_{_{_{-}}}}
\newcommand{\pd}{\cdot}
\newcommand{\ut}{{\rm e}}
\newcommand{\zr}{f}
\newcommand{\eq}{\approx}

\newcommand{\ua}{\mathop{\uparrow}}
\newcommand{\da}{\hspace{.01in} \downarrow \hspace{-.01in}}
\newcommand{\abs}[1]{{ |{#1}|}}

\newcommand{\cs}[1]{{\mathcal C}(\m {#1})}      
\newcommand{\nc}[1]{{\mathcal{NC}}(\m {#1})}    

\newcommand{\Rl}{\mathcal{RL}}
\newcommand{\Irl}{\mathcal{IRL}}
\newcommand{\Prl}{\mathcal{PRL}}

\newcommand{\Semrl}{\mathcal{S}em\mathcal{RL}}

\newcommand{\Semirl}{\mathcal{S}em\mathcal{IRL}}
\newcommand{\LG}{\mathcal{LG}}

\newcommand{\N}{ {\mathbb N} }

\newcommand{\R}{ {\mathbb R} }
\newcommand{\C}{ {\rm C} }
\newcommand{\NC}{ {\rm NC} }
\newcommand{\K}{ {\mathcal K} }

\newtheorem{theorem}{Theorem}[section]
\newtheorem{lemma}[theorem]{Lemma}

\newtheorem{corollary}[theorem]{Corollary}

\newtheorem{proposition}[theorem]{Proposition}

\newtheorem{problem}[theorem]{Problem}
\newtheorem{example}{Example}[section]

\numberwithin{equation}{section}

\newcommand{\prp}[1]{{\sf #1}}

\title[The Conrad Program]{The Conrad Program: From $\ell$-Groups to Algebras of Logic}

\author[M. Botur]{Michal Botur}
\address{Department of Algebra and Geometry\\
Palack\'{y} University in Olomouc \\
Czech Republic}
\email{michal.botur@upol.cz}

\author[J. K\"{u}hr]{Jan K\"{u}hr}
\address{Department of Algebra and Geometry\\
Palack\'{y} University in Olomouc \\
Czech Republic}
\email{jan.kuhr@upol.cz}

\author[L. Liu]{Lianzhen Liu}
\address{School of  Science\\
Jiangnan Universsity \\
China}
\email{lianzhen2003@yahoo.com}

\author[C. Tsinakis]{Constantine Tsinakis}
\address{Department of Mathematics\\
Vanderbilt University\\
U.S.A.}
\email{constantine.tsinakis@vanderbilt.edu}

\keywords{Residuated lattices, Hamiltonian residuated lattices, 
Lattice-ordered groups, GMV-algebras}

\subjclass
[2000] {Primary:
06F05, 
Secondary:
06D35, 
06F15, 
03G10,
03B47, 
08B15, 
}

\begin{document}

\begin{abstract}
A number of research articles have established the significant role of lattice-ordered groups ($\ell$-groups) in logic.  The fact that underpins these studies is the realization that important algebras of logic may be viewed as $\ell$-groups with a modal operator. These connections are just the tip of the iceberg. The purpose of the present article is to lay the groundwork for, and provide  significant initial contributions to, the development of a Conrad type approach to the study of algebras of logic. The term \emph{Conrad Program} refers to Paul Conrad's approach to the study of $\ell$-groups, which analyzes the structure of individual or classes of $\ell$-groups by primarily using strictly lattice theoretic properties of their lattices of convex $\ell$-subgroups. The present article demonstrates that large parts of the Conrad Program can be profitably extended in the setting of $\ut$-cyclic residuated lattices -- that is residuated lattices that satisfy the identity $x\ld \ut\eq \ut\rd x$. An indirect benefit of this work is the introduction of new tools and techniques in the study of algebras of logic, and the enhanced role of the lattice of convex subalgebras of a residuated lattice.
\end{abstract}

\noindent
This is a preprint of an article published in Journal of Algebra.\\
The final publication is available at\\
\url{http://dx.doi.org/10.1016/j.jalgebra.2015.10.015}.

\newpage

\maketitle

\section{Introduction}\label{intro}

There have been a number of studies providing compelling evidence of the importance of lattice-ordered groups ($\ell$-groups) in the study of  algebras of logic\footnote{We use the term \emph{algebra of logic} to refer to residuated lattices -- algebraic counterparts of propositional substructural logics -- and their reducts. Substructural logics are non-classical logics that are weaker than classical logic, in the sense that they may lack one or more of the structural rules of contraction, weakening and exchange in their Genzen-style axiomatization. These logics encompass a large number of non-classical logics related to computer science (linear logic), linguistics (Lambek Calculus), philosophy (relevant logics), and many-valued reasoning.}. For example, a fundamental result \cite{Mun86} in the theory of MV-algebras  is the categorical equivalence between the category of MV-algebras and the category of unital Abelian $\ell$-groups. Likewise, the non-commutative generalization of this result in \cite{Dvu02} establishes a categorical equivalence between the category of pseudo-MV-algebras and the category of unital  $\ell$-groups. Further, the generalization of these two results in \cite{MT10}
shows that one can view GMV-algebras as $\ell$-groups with a suitable modal operator. Likewise, the  work in \cite{MT10} offers a new paradigm for the study of various classes of cancellative residuated lattices by viewing these structures as $\ell$-groups with a suitable modal operator (a conucleus). 

The preceding connections are just the tip of the iceberg. Here we lay the groundwork for, and provide some significant initial contributions to, developing a Conrad type approach to the study of algebras of logic. The term \emph{Conrad Program} traditionally refers to Paul Conrad's approach to the study of $\ell$-groups, which analyzes the structure of individual $\ell$-groups, or classes of $\ell$-groups, by primarily using strictly lattice theoretic properties of their lattices of convex $\ell$-subgroups. Conrad's papers \cites{Con60, Con61, Con65, Con68} in the 1960s pioneered this approach and amply demonstrated its usefulness. A survey of the most important consequences of this approach to $\ell$-groups can be found in \cite{ACM89}, while complete proofs for most of the surveyed results can be found in Conrad's ``Blue Notes'' \cite{Con70}, as well as in \cite{AF88} and \cite{Dar95}.

The present article and its forthcoming successors will demonstrate that large parts of the Conrad Program can be profitably extended in the setting of $\ut$-cyclic residuated lattices -- that is residuated lattices that satisfy the identity $x\ld \ut\eq \ut\rd x$. A byproduct of this work is the addition of new tools and techniques for studying  algebras of logic. 

Let us now be  more specific about the structure and results of this paper. In Section~\ref{basicfacts}, we recall some necessary background from the theory of residuated lattices. In Section~\ref{convexsubs}, we study in detail the lattice $\cs{L}$ of convex subalgebras of an $\ut$-cyclic residuated lattice $\m L$. The main result of this section is Theorem~\ref{CS8} which asserts that $\cs{L}$ is an algebraic distributive lattice whose compact elements -- the principal convex subalgebras -- form a sublattice. Further, Lemma~\ref{CS4} of the same section provides an element-wise description of the convex subalgebra generated by an arbitrary subset of $\m L$. The development of the material of this section is a natural extension of the techniques used to study the lattice of normal convex subalgebras, as developed in~\cite{BT03}. 

In Section~\ref{polarsandprimes}, we consider the role of prime convex subalgebras (meet-irredu\-cible elements) and polars (pseudocomplements) in $\cs{L}$. For example, it is shown (Lemma~\ref{lemma<<primes}) that if the $\ut$-cyclic residuated lattice $\m L$ satisfies the left or right  prelinearity law,  then a convex subalgebra $H$ of $\m L$ is prime iff the set of all convex subalgebras exceeding $H$ is a chain under set-inclusion. Hence, see also Proposition~\ref{relatnormal}, the lattice $\K(\cs{L})$ of principal convex subalgebras of $\m L$ is a relatively normal lattice. Further, a description of minimal prime convex subalgebras in terms of polars is provided by Proposition~\ref{minprime2}. 

Section~\ref{section<<semilinearity} is concerned with semilinearity, and more precisely with the question of whether this property can be ``captured'' in the lattice of convex subalgebras. We prove in Theorem~\ref{semilinear} that a variety  $\vty{V}$ of $\ut$-cyclic residuated lattices that satisfy either of the prelinearity laws is semilinear iff for every $\m L\in\vty{V}$, all (principal) polars in $\m L$ are normal.
Equivalently, all minimal prime convex subalgebras of $\m L$ are normal. Further,
Theorem~\ref{newchar} presents a characterization of the variety $\Semrl$ of semilinear residuated lattices that does  not involve multiplication. 

In Section \ref{section<<hamiltonian}, we study Hamiltonian residuated lattices, that is, residuated lattices in which  convex subalgebras are normal. 
Theorem~\ref{thm<<hamiltonian} characterizes Hamiltonian varieties of \ut-cyclic residuated lattices; more precisely,  a class closed with respect to direct products is Hamiltonian iff it satisfies certain identities.
While it is known that  there exists a largest Hamiltonian variety of $\ell$-groups, viz., the variety of weakly Abelian $\ell$-groups, Theorem~\ref{thm<<hamiltonian2} establishes that  this is not the case for Hamiltonian varieties of $\ut$-cyclic residuated lattices. 

In Section \ref{section<<gmvalgebras} we ask whether the lattice of convex subalgebras of a residuated lattice that satisfies either prelinearity law is isomorphic to the lattice of convex  $\ell$-subgroups of an $\ell$-group. The main result of the section  provides an affirmative answer to the question when the residuated lattice is a GMV-algebra.

Section~\ref{section<<concludingremarks} offers a few suggestions for future developments in the subject.


\section{Basic notions}\label{basicfacts}

In this section we briefly recall basic facts about the varieties of residuated lattices, referring to~\cite{BT03},~\cite{JT02}, ~\cite{GJKO07}, and ~\cite{MPT10} for further details. These varieties provide algebraic semantics for substructural logics, and encompass other important classes of algebras such as $\ell$-groups. 


A \emph{residuated lattice} is an algebra $\m{L} = (L,\cdot ,\ld,\rd,\jn,\mt ,\ut)$ satisfying:
\begin{itemize} 
\item[(a)]	$(L,\cdot,\ut)$ is a monoid;
\item[(b)] 	$(L,\jn,\mt)$ is a lattice with order $\le$; and
\item[(c)] 	$\ld$ and $\rd$ are binary operations satisfying the residuation property: 
		\[
		x\cdot y\leq z \quad{\rm iff }\quad y \leq x\ld z \quad {\rm iff }\quad x\leq z \rd y.
		\]
\end{itemize}
We refer to the operations $\ld$ and $\rd$ as the
{\it left residual} and {\it right residual} of $\cdot$, respectively.
As usual, we write $xy$ for $x\cdot y$ and adopt the convention that, in
the absence of parenthesis, $\cdot$ is performed first, followed by
$\ld$ and $\rd$, and finally by $\jn$ and $\mt$.

Throughout this paper, the class of residuated lattices will be
denoted by $\Rl$. It is easy to see that the equivalences that define residuation can be expressed in terms of finitely many equations and thus $\Rl$ is a finitely based variety (see~\cite{BT03},~\cite{BCGJT03}).   

The existence of residuals has the following basic consequences,
which will be used in the remainder of the paper without explicit
reference.

\begin{lemma}\label{residuals}
Let $\m L$ be a residuated lattice.
\begin{enumerate}
\item[\rm{(1)}]
The multiplication preserves all existing joins in each argument;
i.e., if $\bigvee X$ and $\bigvee Y$ exist for $X, Y\subseteq L,$ then
$\bigvee_{x\in X, y\in Y}(xy)$ exists and
\[
\Big(\bigvee X\Big)\Big(\bigvee Y\Big)=\bigvee_{x\in X, y\in Y}(xy).
\]
\item[\rm{(2)}]
The residuals preserve all existing meets in the numerator, and
convert existing joins to meets in the denominator, i.e.,
if $\bigvee X$ and $\bigwedge Y$ exist for $X,Y\subseteq L$, then for
any $z\in L$, $\bigwedge_{x\in X}(x\ld z)$ and
$\bigwedge_{y\in Y}(z\ld y)$ exist and
\[
\Big(\bigvee X\Big)\Big\backslash z=\bigwedge_{x\in X}(x\ld z)
\quad\text{ and }\quad
z\,\Big\backslash\Big(\bigwedge Y\Big)=\bigwedge_{y\in Y}(z\ld y)\text{,}
\]
and the same for $\rd$.

\item[\rm{(3)}] The following identities (and their mirror images)\footnote{(a), (b) and (c)  are expressed as inequalities, but are clearly equivalent to identities.}
hold in $\m L$:
\begin{enumerate}
\item[\rm{(a)}] $(x\ld y)z\leq x\ld yz$;
\item[\rm{(b)}] $x\ld y\leq zx\ld zy$;
\item[\rm{(c)}] $(x\ld y)(y\ld z)\leq x\ld z$;
\item[\rm{(d)}] $xy\ld z=y\ld (x\ld z)$;
\item[\rm{(e)}] $x\ld (y/ z)=(x\ld y)/ z$;
\item[\rm{(f)}] $x(x\ld x)=x$;
\item[\rm{(g)}] $(x\ld x)^2=x\ld x$.
\end{enumerate}
\end{enumerate}
\end{lemma}

We will have the occasion to consider pointed residuated lattices. A \emph{pointed residuated lattice}  is  an algebra  $\m{L} = (L,\cdot ,\ld,\rd,\jn,\mt ,\ut,  \zr)$ of signature $\lan 2, 2, 2, 2,$ $2, 0, 0\ran$ such that $(L,\cdot ,\ld,\rd,\jn,\mt ,\ut)$ is a residuated lattice. In other words, a pointed residuated lattice is simply a residuated lattice with an extra constant $\zr$. Residuated lattices may be identified with pointed residuated lattices satisfying the identity 
$\ut \eq \zr$. 
The variety $\Prl$ of pointed residuated lattices provides algebraic semantics for the Full Lambek calculus (pointed residuated lattices are therefore often referred to also as FL-algebras) and its subvarieties  correspond to substructural logics.
We also define here  a \emph{bounded residuated lattice} to be a pointed residuated lattice with bottom element $\zr$ (and therefore also, top element $\zr\ld\zr$), emphasizing that ``bounded'' implies that the constant $\zr$ representing the bottom element is in the signature. 
A  {(pointed)} residuated lattice is said to be {\em integral} if $\ut$ is its top element. {Note, however, that in a bounded residuated lattice, $\ut$ may not be the top element, 
and, conversely, an integral pointed residuated lattice may not be bounded.}
A {(pointed)} residuated lattice is \emph{commutative} if it 
satisfies the equation $xy \eq yx$, in which case, $x\ld y$ and $y\rd x$ coincide and are denoted by $x\ra y$. We call a {(pointed)} residuated lattice \emph{$\ut$-cyclic} if it satisfies the identity  $\ut\rd x\eq x\ld\ut$. In this paper we mostly consider $\ut$-cyclic {(pointed)} residuated lattices.

Important particular classes of $\ut$-cyclic residuated lattices are $\ell$-groups, GBL-algebras and GMV-algebras. The latter two classes will be discussed in some detail in Section \ref{section<<gmvalgebras}.

The variety of $\ell$-groups occupies a very special place among varieties of residuated lattices. Recall that an element $a\in L$ is said to be \emph{invertible} if $(\ut\rd a)$ $a=\ut=a(a\ld \ut)$. This is of course true if and only if $a$ has a (two-sided) inverse $a^{-1}$, in which case $\ut\rd a=a^{-1}=a\ld \ut$. The structures in which every element is invertible are precisely the $\ell$-groups. It should be noted that an $\ell$-group is traditionlly
defined in the literature as an algebra 
$\m G = (G, \jn, \mt, \cdot, { }^{-1}, \ut)$ such that $(G, \jn, \mt)$ is a
lattice, $(G, \cdot , { }^{-1}, \ut)$ is a group, and multiplication
is order preserving -- or, equivalently, it distributes over the
lattice operations (see~\cite{AF88},~\cite{Gla99}). 
The variety of $\ell$-groups is term equivalent to the subvariety of $\Rl$ defined by the
equations $(\ut \rd x) x \eq \ut \eq x(x\ld \ut);$
the term equivalence is given
by $x^{-1}=\ut \ld x$ and $x \rd y = x y^{-1}, \: x \ld y = x
^{-1}y$. We denote by $\LG$ the aforementioned subvariety
and refer to its members as $\ell$\emph{-groups}. 

Let $\m{L}$ be a residuated lattice.  If $F\subseteq L$, we write $F^-$ for the set of ``negative'' elements of $F$; i.e., 
$F^-=\{x \in F \colon  x\leq \ut\}$. The \emph{negative cone} of $\m{L}$ is the algebra $\m{L}^{-}$ with domain $L^-$, 
with monoid and lattice operations  the restrictions to $L^{-}$ of the corresponding operations in $\m{L}$, and with the residuals $\ldm$  and $\rdm$ defined by
\[
x\ldm y=(x\ld y)\mt \ut \quad \text{and}\quad y \rdm x=(y\rd x)\mt \ut,
\]
where $\ld $ and $\rd$ denote the residuals in $\m{L}$.


\section{Convex subalgebras}\label{convexsubs}

In many respects, convex $\ell$-subgroups play a more important role than $\ell$-ideals (normal  convex $\ell$-subgroups) in the study of $\ell$-groups, for the lattice of convex $\ell$-subgroups of a given $\ell$-group bears significant information about the $\ell$-group. Some investigations similar to the aforementioned ``Conrad program'' have been carried out for prefilters (or filters) of certain integral residuated lattices and for convex subalgebras of GBL-algebras -- more precisely, for ideals of DR$\ell$-monoids. 
In the context of $\ut$-cyclic residuated lattices, convex subalgebras are a natural generalization of convex $\ell$-subgroups of $\ell$-groups as well as prefilters (filters) of integral residuated lattices. Hence we will focus on convex subalgebras of $\ut$-cyclic residuated lattices and on the study of $\ut$-cyclic residuated lattices via their lattices of convex subalgebras.

In this section, we describe some basic properties of convex subalgebras of $\ut$-cyclic residuated lattices and we show that the lattice of convex subalgebras of $\m L$ is an algebraic distributive lattice which is isomorphic to a sublattice of the congruence lattice of the lattice reduct of $\m L$.

A subset $C$ of a poset $\m P = (P,\leq)$ is \emph{order-convex} (or simply
\emph{convex}) in $\m P$ if, whenever $a,b\in C$ with $a\leq b$, then
$\ua a \ \cap \da\! b \subseteq C$. For a  residuated lattice $\m L$, we write $\cs{L}$ for the set of all convex
subalgebras of $\m L$, partially ordered by set-inclusion.  It is easy to see that the
intersection of any family of convex subalgebras of $\m L$ is again a convex subalgebra of
$\m L$, and so is the union of any up-directed family of convex 
subalgebras. Thus, $\cs{L}$ is an algebraic closure system.

For any $S \subseteq L$, we  let $\C[S]$ denote the smallest convex subalgebra of $\m L$ containing $S$.  As is customary, we call $\C[S]$ the convex subalgebra \emph{generated} by $S$ and  let $\C[a] = \C[\{a\}]$. We refer to $\C[a]$ as the \emph{principal} convex subalgebra of $\m L$ generated by the element $a$.  We say that a convex subalgebra of $\m L$ is \emph{finitely generated} provided it is generated by some finite subset of $L$. 

Given a residuated lattice $\m L$ and an element $x\in L$, the \emph{absolute value} of $x$ is the element 
\[
\abs{x}=x\mt (\ut\rd x) \mt \ut.
\]
If $X\subseteq L$, we set $\abs{X}=\{\abs{x}\colon x\in X\}$.
We note that in the case of GBL-algebras, $\abs{x}=x\mt (\ut\rd x)$.

The proof of the following lemma is routine:

\begin{lemma}\label{absvalue}
Let $\m L$ be an $\ut$-cyclic residuated lattice, let  $x\in L$, and let $a\in L^-$. The following conditions hold:
\begin{enumerate}
\item[\rm{(1)}] $x\leq\ut$ iff $\abs{x}=x$;
\item[\rm{(2)}] $\abs{x}\leq x\leq \abs{x}\ld\ut$;
\item[\rm{(3)}]  $\abs{x}=\ut$ iff $x=\ut$;
\item[\rm{(4)}]  $a\leq x\leq a\ld\ut$ iff $a\leq \abs{x}$; and
\item[\rm{(5)}]  if $H\in\cs{L}$, then $x\in H$ iff $\abs{x}\in H$.
\end{enumerate}
\end{lemma}

In what follows, given a residuated lattice $\m L$ and a subset $S\subseteq L$, we write $ \lan S \ran$ for the submonoid of $\m L$ generated by $S$. Thus, $x\in \lan S \ran$ if and only if there exist elements $s_1,\dots, s_n\in S$ such that $x=s_1\cdots s_n$.

\begin{lemma}\label{CS4} 
Suppose $\m L$ is an $\ut$-cyclic residuated lattice and let $S \subseteq L$. 
Then 
\begin{align*}
\C[S]  = \C[\abs{S}] & = \{ x\in L \colon h \leq x \leq h \ld \ut, \text{ for some } h \in \lan\abs{S}\ran \} \\
             & = \{ x\in L \colon h \leq \abs{x}, \text{ for some } h \in \lan\abs{S}\ran \}. 
\end{align*}
\end{lemma}

\begin{proof}
It is clear, in view of the definition of $\abs{.}$ and Condition (2) of Lemma \ref{absvalue}, that $\C[S] = \C[\abs{S}]$. Hence, by Condition (4) of Lemma \ref{absvalue}, it will suffice to prove that 
$\C[\abs{S}]=\{ x\in L \colon h \leq x \leq h \ld \ut, \text{ for some } h \in \lan\abs{S}\ran \}$. 
Set $H=\{ x\in L \colon h \leq x \leq h \ld \ut, \text{ for some } h \in \lan\abs{S}\ran \}$. It is clear that $H \subseteq \C[\abs{S}]$. We complete the proof by showing that $H$ is a convex subalgebra of $\m L$ that contains $\abs{S}$. Let $s \in \abs{S}$. 
Then $s \leq s \leq s \ld \ut$ since $s \leq \ut$, 
and so $\abs{S} \subseteq H$. Let $x,y \in H$. There exist  
$h,k \in\lan \abs{S} \ran\subseteq L^-$ such that
$h \leq x \leq h \ld \ut$ and $k \leq y \leq k \ld \ut$. It  follows that 
$(hk)(kh) \leq hk \leq xy \leq (h \ld \ut)(k \ld \ut) \leq kh
\ld \ut \leq (hk) (kh) \ld \ut$. This shows that $xy \in H$, and whence $H$ is closed under multiplication. Also, $hk \leq h\mt k \leq x\mt y \leq x\jn y  \leq
(h \ld \ut) \jn (k \ld \ut) \leq hk \ld \ut$.  Thus, 
$x \mt y, x \jn y \in H$, and $H$ is closed under the lattice operations. 

To demonstrate closure with respect to residuals, observe that $y\leq k\ld\ut=$ $\ut\rd k$ implies $ykh\leq h\leq x$ whence $hkh\leq kh\leq y\ld x$, and on the other hand, $k\leq y$ and $x\leq h\ld\ut$ imply 
$y\ld x \leq k\ld x \leq k\ld (h\ld\ut)=hk\ld\ut\leq hkh\ld\ut$. Thus $hkh\leq y\ld x\leq hkh\ld\ut$.
Analogously, $khk\leq x\rd y\leq khk\ld\ut$ because 
$y\leq k\ld\ut$ implies $hky\leq h\leq x$, so $khk\leq hk\leq x\rd y$,
and also $x\rd y\leq x\rd k\leq (\ut\rd h)\rd k=$ $\ut\rd kh\leq khk\rd\ut$.

Order-convexity is almost immediate. Indeed, let $z\in L$ such that $x \leq z \leq y$. Then $hk \leq h \leq x \leq z \leq y \leq k \ld \ut \leq hk \ld \ut$. Thus, $z \in H$ and  $H$ is convex.
\end{proof}

The next two corollaries are immediate consequences of the preceding result:

\begin {corollary}\label{CS1} 
Suppose $\m L$ is an $\ut$-cyclic residuated lattice, and let $a\in L$. Then 
\begin{align*}
\C[a] = \C[\abs{a}] &=  \{ x\in L\colon \abs{a}^n \leq x \leq \abs{a}^n \ld \ut, \text{ for some } n\in\N \} \\
& = \{ x\in L\colon \abs{a}^n \leq \abs{x}, \text{ for some } n\in\N \}.
\end{align*}
\end {corollary}

\begin{corollary}\label{CS6} 
Let $\m L$ be an $\ut$-cyclic residuated lattice, and let $H$ be a convex subalgebra of $\m L$. Then $H=\C[H^-]$.
\end{corollary}

Recall that a prefilter (sometimes also called filter) of an integral residuated lattice is a non-empty upwards closed subset that is closed under multiplication (e.g. \cite{vA02}). In the integral case, it is quite obvious that prefilters are exactly convex subalgebras. Using absolute values and Lemma \ref{CS4} we can similarly describe convex subalgebras also in the non-integral case.
Prefilters defined below are dual notions to ideals of DR$\ell$-monoids, which, as noted earlier,  are the duals of GBL-algebras; ideal lattices of these algebras were studied in \cite{Kuhr03a}, \cite{Kuhr03b} and \cite{Kuhr05}.

We say that a non-empty subset $P$ of an $\ut$-cyclic residuated lattice $\m L$ is a \emph{prefilter} of $\m L$ if 
\begin{itemize}
\item $xy\in P$ for all $x,y\in P$; and 
\item for all $x\in P$ and $y\in L$, $\abs{x}\leq\abs{y}$ implies $y\in P$.
\end{itemize}
Obviously, if $P$ is a prefilter of $\m L$, then $P^-$ is a prefilter of the integral residuated lattice $\m L^-$.
Further, the second condition entails that $x\in P$ iff $\abs{x}\in P$, for all $x\in L$.

\begin{proposition}
Let $\m L$ be an $\ut$-cyclic residuated lattice, and $P\subseteq L$.
Then $P$ is a prefilter of $\m L$ iff it is a convex subalgebra of $\m L$. 
\end{proposition}

\begin{proof}
Let $P$ be a prefilter of $\m L$. If $x\in\C[P]$, then by Lemma \ref{CS4}, there exist $p_1,\dots,p_n\in\abs{P}$ such that $p_1\dots p_n \leq \abs{x}$. 
As noted above, $\abs{P}\subseteq P$, and since $P$ is closed under multiplication, we have that $p_1\dots p_n\in P$. In other words, $p\leq\abs{x}$ for some $p\in P$, whence $x\in P$. Thus $\C[P]\subseteq P$ and we conclude that $P=\C[P]$.

Conversely, recalling Lemma \ref{absvalue} (5), we see that every convex subalgebra is a prefilter.
\end{proof}

The next result shows that every finitely generated member of the algebraic closure system $\cs{L}$ is principal, and that the poset of principal convex subalgebras of $\m L$ is a sublattice of $\cs{L}$, denoted by $\K(\cs{L})$.

\begin{lemma}\label{CS3} 
Let $\m L$ be an $\ut$-cyclic residuated lattice, and  $a,b \in L$. 
Then $\C[a]\cap \C[b]=\C[\abs{a}\jn \abs{b}]$ and $\C[a] \jn \C[b] = \C[\abs{a} \mt \abs{b}]$ $=\C[\abs{a}\abs{b}]$.
\end{lemma}

\begin{proof} 
In view of Lemma \ref{absvalue} and Corollary \ref{CS1}, we may assume  that $a,b \in L^-$. We first establish the equality 
$\C[a]\cap \C[b]=\C[a\jn b]$. We have $a,b\leq a \jn b \leq \ut$ and so $\C[a\jn b]
\subseteq \C[a]\cap \C[b]$. Now, let $x\in \C[a] \cap \C[b]$. 
Then by Corollary \ref{CS1}, $a ^ m \leq x \leq a ^ m \ld \ut$
and $b ^ n \leq x \leq b ^ n \ld \ut$ for $m,n \in \N$. Note that
$(a \jn b) ^ {mn} \leq a ^ m \jn b ^ n$.
Hence, $(a \jn b) ^ {mn} \leq a ^ m \jn b ^ n \leq x \leq (a ^ m
\ld \ut) \mt (b ^ n \ld \ut)=(a ^ m \jn b ^ n) 
\ld \ut \le (a \jn b) ^ {mn} \ld \ut$. It follows that $x \in C[a \jn b]$, and hence the first equality holds.

We next show that $\C[a] \jn \C[b] = \C[a \mt b] = \C[ab]$. Since $ab \leq a \mt b \leq a,b \leq \ut$, 
we have $\C[a] \jn \C[b] \subseteq \C[a\mt b] \subseteq \C[ab]$. Also $ab$ must belong to $\C[a] \jn \C[b]$. So, $\C[ab] \subseteq \C[a] \jn \C[b]$.
\end{proof}

The following trivial fact will used quite often in calculations:

\begin{lemma}\label{triviallemma}
Let $x_1,\dots,x_n,y\in L^-$ in a residuated lattice $\m L$. Then
\[
(x_1\jn y)\dots (x_n\jn y)\leq x_1\dots x_n \jn y.
\]
\end{lemma}
\begin{proof}
Indeed, $(x_1\jn y)(x_2\jn y)=x_1x_2 \jn x_1y \jn yx_2 \jn y^2\leq x_1x_2 \jn y$ for $n=2$, and the rest is an easy induction.
\end{proof}

This leads to the main result of this section:

\begin{theorem}\label{CS8} 
If $\m L$ is an $\ut$-cyclic residuated lattice, then $\cs{L}$ is a distributive algebraic  lattice. The poset $\K(\cs{L})$ of compact elements of $\cs{L}$ -- consisting of the principal convex subalgebras of $\m L$ -- is a sublattice of $\cs{L}$. 
\end{theorem}

\begin{proof}
Let $\m L$ be an $\ut$-cyclic residuated lattice. In view of the preceding results, we only need to prove that the lattice $\cs{L}$ is distributive. To this end, let $X,Y,Z\in\cs{L}$, and let $x\in X\cap (Y\jn Z)$. Then $x\in X$ and there exist $a_1,\dots,a_n\in Y^- \cup Z^-$ such that $a_1\dots a_n\leq \abs{x}$. Invoking the preceding lemma we get 
\[
(a_1\jn \abs{x})\dots (a_n\jn \abs{x})\leq a_1\dots a_n \jn\abs{x}=\abs{x}.
\]
Since $\abs{x}\in X^-$, the element $a_i \jn \abs{x}$ belongs to $(X\cap Y)^-$ when $a_i\in Y^-$, or to $(X\cap Z)^-$ when $a_i\in Z^-$. Therefore, $\abs{x}$, and hence, by Lemma \ref{absvalue}, $x$ is an element of $\C[(X\cap Y)\cup (X\cap Z)]=(X\cap Y)\jn (X\cap Z)$.
This proves $X\cap (Y\jn Z) \subseteq (X\cap Y)\jn (X\cap Z)$, and the distributivity of $\cs{L}$.
\end{proof}

Since the variety of lattices is congruence-distributive, the distributivity of the lattice $\cs{L}$ also follows from the following:

\begin{theorem}
Let $\m L$ be an $\ut$-cyclic residuated lattice. Then the lattice $\cs{L}$ can be embedded (as a complete sublattice) into the congruence lattice of the lattice reduct of $\m L$.
\end{theorem}

\begin{proof}
For every $H\in\cs{L}$, 
$\Theta_H=\{\lan a,b\ran \in L^2\colon (a\ld b) \mt (b\ld a) \mt\ut \in H\}$
is a congruence of the lattice reduct $(L,\jn,\mt)$ of $\m L$ such that $[\ut]_{\Theta_H} = H$.
Indeed, it is proved in \cite{BT03} and \cite{JT02} that $\Theta_H$ is a congruence of $\m L$ provided that $H$ is a normal\footnote{Briefly, $H\in\cs{L}$ is \emph{normal} iff for all $a,b\in L$, $(a\ld b)\mt\ut \in H$ iff $(b\rd a)\mt\ut\in H$. Normal convex subalgebras are discussed in Section \ref{section<<semilinearity}.} convex subalgebra, but normality of $H$ is not used in proving that $\Theta_H$ is an equivalence relation which is compatible with the lattice operations.

It is easily seen that $H\subseteq K$ iff $\Theta_H \subseteq \Theta_K$, for all $H,K\in\cs{L}$.
Thus the map $H\mapsto\Theta_H$ is an order-embedding of $\cs{L}$ into the congruence lattice of $(L,\jn,\mt)$, the lattice reduct of $\m L$.

Let $\{H_i\colon i\in I\}$ be a collection of convex subalgebras of $\m L$, and let $H=\bigcap_{i\in I} H_i$ and $K=\bigvee_{i\in I} H_i$ in $\cs{L}$.
Writing $\Theta_i$ for $\Theta_{H_i}$, we obviously have $\bigcap_{i\in I} \Theta_i = \Theta_{H}$ and $\bigvee_{i\in I} \Theta_i \subseteq \Theta_K$.

Let $\lan a,b\ran \in\Theta_K$. Since $\Theta_K$ is a lattice congruence, we may safely assume that $a\leq b$.
Then $(b\ld a) \mt\ut \in K$, so there exist $x_1,\dots,x_m\in \bigcup_{i\in I} H_i^-$ such that $x_1\dots x_m\leq (b\ld a) \mt\ut$, whence $bx_1\dots x_m\leq a$.
Then
\[
x_m\leq (bx_1\dots x_{m-1}\ld a) \mt\ut \leq \big((a\jn bx_1\dots x_{m-1})\ld a\big) \mt\ut,
\]
where the latter inequality follows from the fact that all residuated lattices satisfy the inequality $(x\ld y) \mt \ut \leq ((z\jn x)\ld (z\jn y)) \mt\ut$.
Hence, letting $z_1=a\jn bx_1\dots x_{m-1}$, we have $a\leq z_1$ and $(z_1\ld a) \mt\ut \in \bigcup_{i\in I} H_i$, so $\lan a,z_1\ran \in \bigcup_{i\in I}\Theta_i$.
Analogously,
\begin{align*}
x_{m-1} & \leq (bx_1\dots x_{m-2} \ld bx_1\dots x_{m-2}x_{m-1})\mt\ut \\
& \leq \big((a\jn bx_1\dots x_{m-2}) \ld (a\jn bx_1\dots x_{m-2}x_{m-1})\big)\mt\ut,
\end{align*}
and letting $z_2=a\jn bx_1\dots x_{m-2}$, we get $z_1\leq z_2$ and $(z_2\ld z_1) \mt\ut \in \bigcup_{i\in I} H_i$.
Thus $\lan z_1,z_2\ran \in \bigcup_{i\in I} \Theta_i$.
Repeating this procedure we get
\[
x_1\leq (b\ld bx_1)\mt\ut \leq (b\ld (a\jn bx_1))\mt\ut,
\]
so with $z_{m-1}=a\jn bx_1$ we have $z_{m-1}\leq b$ and $(b\ld z_{m-1})\mt\ut \in \bigcup_{i\in I} H_i$. Thus $\lan z_{m-1},b\ran \in\bigcup_{i\in I}\Theta_i$.

We have found $z_0,\dots,z_{m}\in L$ such that $a=z_0\leq z_1\leq\dots\leq z_{m-1}\leq z_m=b$ and $\lan z_j,z_{j+1}\ran \in\bigcup_{i\in I}\Theta_i$ for $j=0,\dots,m-1$, which means that $\lan a,b\ran \in\bigvee_{i\in I}\Theta_i$.
Hence $\Theta_K\subseteq\bigvee_{i\in I}\Theta_i$ and the proof is complete.
\end{proof}


\section{Polars and prime convex subalgebras}\label{polarsandprimes}

We have seen in the previous section that the lattice $\cs{L}$ of convex subalgebras of an $\ut$-cyclic residuated lattice $\m L$ is an algebraic distributive  lattice. As such, it satisfies the join-infinite distributive law
\[
X \,\cap \bigvee_{i\in I}Y_i=\bigvee_{i\in I}(X\cap\, Y_i),
\]
and hence it is relatively pseudocomplemented.  
That is, for all $X,Y\in\cs{L}$, the relative pseudocomplement $X\ra Y$ of $X$ relative to $Y$ exists:
\[
X\ra Y =\max\{Z\in \cs{L}\colon X\cap Z\subseteq Y\}.
\]
The next lemma provides an element-wise description of $X\ra Y$ in terms of the absolute value, and in particular one for the pseudocomplement 
$X^\perp =X\ra \{\ut\}$  of $X$.

\begin{lemma}\label{polars}
If $\m L$ is an $\ut$-cyclic residuated lattice, then $\cs{L}$ is a relatively pseudocomplemented lattice. Specifically, given $X,Y\in\cs{L}$,
\[
X\ra Y=\{a\in L\colon \abs{a} \jn \abs{x} \in Y, \text{ for all } x\in X\}\text{,}
\]
and in particular, 
\[
X^\perp=X\ra\{\ut\}=\{a\in L\colon \abs{a}\jn \abs{x}=\ut, \text{ for all } x\in X\}.
\] 
\end{lemma}

\begin{proof}
As noted above, $\cs{L}$ is relatively pseudocomplemented. Thus the proof is a direct consequence of Lemmas \ref{absvalue} (5) and  \ref{CS3}. Indeed, let $X,Y\in\cs{L}$ and $a\in L$. Then $a\in X\ra Y$ iff $\C[a]\subseteq X\ra Y$ iff $\C[a]\cap X\subseteq Y$ iff $\C[a]\cap \C[x]\subseteq Y, \text{ for all } x\in X$ iff $\C[\abs{a}\jn\abs{x}]\subseteq Y, \text{ for all } x\in X$ iff $\abs{a}\jn\abs{x}\in Y, \text{ for all } x\in X$. The description of $X^\perp$ is clear.
\end{proof}

We can define $X^\perp$ for any non-empty subset $X\subseteq L$. It is not hard to show that $X^\perp=\C[X]^\perp$, so $X^\perp$ is always a convex subalgebra.
We refer to $X^\perp$ as the \emph{polar} of $X$; in case $X=\{x\}$, we write $x^\perp$ instead of $\{x\}^\perp$ (or $\C[x]^\perp$) and refer to it as the \emph{principal polar} of $x$.

The map $^\perp\colon \cs{L}\ra \cs{L}$ is a self-adjoint inclusion-reversing map, while the map sending $H\in \cs{L}$ to its double polar $H^{\perp\perp}$ is an intersection-preserving closure operator on $\cs{L}$. By a classic
result due to Glivenko, the image of this closure operator is a (complete) Boolean
algebra $\m{B}_{\cs{L}}$ with least element $\{\ut\}$ and largest element
$L$. The complement of $H$ in $\m{B}_{\cs{L}}$ is precisely $H^\perp$, whereas, for any pair of convex subalgebras $H, K\in \m{B}_{\cs{L}}$, 
\[
H\jn^{\m{B}_{\cs{L}}}K=(H^\perp\cap K^\perp)^\perp \, =(H\cup K)^{\perp\perp}.
\]
On the other hand, meets in $\m{B}_{\cs{L}}$ are just intersections.

Two identities of particular interest to us are the \emph{left prelinearity law} \prp{LP} and the \emph{right prelinearity law} \prp{RP}:
\begin{gather*}
((x\ld y) \mt\ut) \jn ((y\ld x)\mt\ut) \eq \ut,\\
((y\rd x) \mt\ut) \jn ((x\rd y)\mt\ut) \eq \ut.
\end{gather*}

Recall that $H\in \cs{L}$ is said to be \emph{prime} if it is meet-irreducible in $\cs{L}$. That is, whenever $X, Y\in \cs{L}, \text{ and } X\cap Y\subseteq H, \text{ then }  X\subseteq H$ or $Y\subseteq H$. Either prelinearity law implies that the poset of prime convex subalgebras of $\m L$ is a root system (dual tree), as the next result shows. 

The next lemma gives a characterization of prime convex subalgebras.

\begin{lemma}\label{lemma<<primes}
Let $\m L$ be an $\ut$-cyclic residuated lattice that satisfies \prp{LP}. 
Then for every $H\in\cs{L}$, the following are equivalent:
\begin{enumerate}
\item[\rm{(1)}] $H$ is a prime convex subalgebra of $\m L$.
\item[\rm{(2)}] For all $a,b\in L$, if $\abs{a}\jn \abs{b}\in H$, then $a\in H$ or $b\in H$.
\item[\rm{(3)}] For all $a,b\in L$, if $\abs{a}\jn \abs{b}=\ut$, then $a\in H$ or $b\in H$.
\item[\rm{(4)}] For all $a,b\in L$, $(a\ld b) \mt\ut \in H$ or $(b\ld a) \mt\ut \in H$.
\item[\rm{(5)}] The set of all convex subalgebras exceeding $H$ is a chain under set-inclusion.
\end{enumerate}
\end{lemma}

\begin{proof} 
(1) implies (2): If $\abs{a} \jn \abs{b} \in H$, then $\C[a] \cap \C[b] = \C[\abs{a} \jn \abs{b}] \subseteq H$, whence $\C[a]\subseteq H$ or $\C[b]\subseteq H$, so $a\in H$ or $b\in H$, by Lemma \ref{CS3}.

(2) implies (3): It follows by specialization.

(3) implies (4): Since $((a\ld b)\mt\ut) \jn ((b\ld a)\mt\ut) = \ut$ by \prp{LP}, it follows that $(a\ld b) \mt\ut \in H$ or $(b\ld a) \mt\ut \in H$.

(4) implies (5): Suppose that $X$ and $Y$ are two incomparable convex subalgebras that exceed $H$. Then there exist $a\in X\setminus Y$ and $b\in Y\setminus X$ and we may assume that $a,b\in L^-$. If $(a\ld b)\mt\ut\in H\subseteq X$, then $a((a\ld b)\mt\ut)\in X$ which implies $b\in X$ because $a((a\ld b)\mt\ut)\leq b\leq\ut$. Analogously, if $(b\ld a)\mt\ut\in H\subseteq Y$, then $b((b\ld a)\mt\ut)\in Y$ whence $a\in Y$. In either case we reach a contradiction.

(5) implies (1): Obvious.
\end{proof}

It is worth noting that Conditions \rm{(1)} and \rm{(4)} are equivalent if and only if $\m L$ satisfies \prp{LP}.
Indeed, suppose that $\m L$ does not satisfy \prp{LP}, 
i.e., there exist $a,b\in L$ such that $s=((a\ld b)\mt\ut)\jn ((b\ld a)\mt\ut)<\ut$. Now, if $V$ is a value\footnote{A value of an element $a\in L$  is a (necessarily completely meet-irreducible) convex subalgebra that is maximal with respect to not containing $a$.} 
of $s$, then clearly neither $(a\ld b)\mt\ut$ nor $(b\ld a)\mt\ut$ belongs to $V$ (for otherwise $s\in V$).
Thus $V$ is a prime convex subalgebra which does not fulfill Condition  \rm{(4)}.

We also note that the right prelinearity law \prp{RP} can be used to obtain the dual of Lemma \ref{lemma<<primes}. One simply replaces $\ld$ by $\rd$ in Condition \rm{(4)}.

Lemma \ref{lemma<<primes} has the following important consequence:

\begin{corollary}\label{linearity}
Let $\m L$ be an $\ut$-cyclic residuated lattice that satisfies either prelinerity law. If $\cs{L}$ -- equivalently, $\K(\cs{L})$ -- is totally ordered, then so is $\m L$.
\end{corollary}

A lower-bounded distributive lattice $\m A = (A,\jn,\mt, \bot)$ is said to be \emph{relatively normal} provided its prime ideals form a root-system under set-inclusion. Relatively normal lattices form an important class of lattices.  They include dual relative Stone lattices, the lattice of cozero-sets of any topological space, and the lattices of compact congruences for many well-studied ordered algebraic structures such as semilinear $\ell$-groups, Reisz spaces, $f$-rings, etc.  (The reader is referred to \cite{ST95} and \cite{HT94} for details and an extensive bibliography on this class.) Lemma~\ref{lemma<<primes} implies that the lattice $\K(\cs{L})$ of principal convex subalgebras of an $\ut$-cyclic residuated lattice that satisfies either prelinearity law is relatively normal. 

The next result, due essentially to Monteiro \cites{Mon54, Mon74}, catalogues several conditions that are equivalent to relative normality. 
Its easy proof is left to the reader.  

\begin{proposition} For a lower-bounded, distributive lattice $\m A$, the
following are equivalent:
\begin{enumerate}
\item[\rm{(1)}] $\m A$ is relatively normal;
\item[\rm{(2)}] Any two incomparable prime ideals in $\m A$ have disjoint open neighborhoods in the Stone space of $\m A$;
\item[\rm{(3)}] The join of any pair of incomparable prime filters in $\m A$ is all of $A$; and
\item[\rm{(4)}] For all $a,b \in A$, there exist $u,v \in A$ such that $u\mt v = \bot$ and $u\jn b =a \jn b = a\jn v$.
\end{enumerate}
\end{proposition}

Condition (4) is a key property for the structure of relatively normal lattices. While the relative normality of $\K(\cs{L})$ follows from Lemma \ref{lemma<<primes}, we can use Condition (4) to provide an alternative proof of this fact.

\begin{proposition}\label{relatnormal}
Let $\m L$ be an $\ut$-cyclic residuated lattice that satisfies either prelinerity law. Then $\K(\cs{L})$ is a relatively normal lattice.
\end{proposition}

\begin{proof} Without loss of generality, we will assume that $\m L$ satisfies \prp{LP}. We have already established  that $\K(\cs{L})$ is a lower-bounded distributive lattice, whose elements are the principal convex subalgebras of $\m L$. Let $X,Y \in \K(\cs{L})$.  By Corollary \ref{CS1}, there exist $a,b\in L^-$ such that $X = \C[a]$ and $Y =\C[b]$.  
Let $u = (b \ld a)\mt\ut$ and $v = (a \ld b)\mt\ut$.
We know by Lemma \ref{CS3} that $\C[u] \cap \C[v] = \C[u \jn v]$.
Consequently, since $u \jn v = ((a \ld b)\mt\ut) \jn ((b \ld a)\mt\ut) = \ut$, 
we see that $\C[u] \cap \C[v] = \{\ut\}$.  

To complete the proof, we need to show that $\C[a] \jn \C[v] = \C[a] \jn \C[b] = \C[u] \jn \C[b]$.
We just prove that $\C[a] \jn \C[v] = \C[a] \jn \C[b]$; the proof of the other equality is analogous.

To prove that $\C[a] \jn \C[v] \subseteq \C[a] \jn \C[b]$, it suffices, by
Lemma \ref{CS3}, to prove that $a \mt v \in \C[a] \jn \C[b]$.  Since $a \leq \ut$ by assumption, we know $a \mt v = a \mt (a \ld b)$.  By definition, $\C[a] \jn \C[b] = \C[\{a,b\}]$; hence, it is clear that $a \land v \in \C[a] \jn \C[b]$. 

To obtain the reverse inclusion, it suffices by Lemma \ref{CS3}  to prove
that $a \mt b \in \C[a] \jn \C[v]$.  Observe that
$(a \mt v)^2 = (a\mt(a\ld b))^2\leq a^2 \leq a,$ and
$(a \mt v)^2 = (a\mt(a\ld b))^2\leq a(a\ld b)\leq b$. It follows that $(a \mt v)^2 \leq a \mt b$.  Since $a \mt b \leq \ut$ by assumption,
convexity now implies that $a \mt b \in \C[a] \jn \C[v]$.
\end{proof}

We close this section by discussing the relationship of principal polars with minimal prime convex subalgebras. A result that has appeared in a number of papers, ranging from $\ell$-groups to MV-algebras and other more general classes of bounded residuated lattices, states that if $H$ is a minimal prime convex subalgebra (called filter or prefilter in this case) of such an algebra, then $H=\bigcup\{x^\perp\colon x\not\in H\}.$ We want to draw attention to \cite{Tsi79} where it is shown that this is truly a result about algebraic distributive lattices, and has little to do with particular properties of the algebras in question. For the convenience of the reader, we provide the details of this relationship.

In the remainder of this section, $\m A$ will denote a nontrivial algebraic distributive lattice whose poset $\K(\m A)$ of compact elements forms a sublattice. The top and bottom elements of $\m A$ will be denoted by $\top$ and $\bot$, respectively. We use $a\ra b$ to denote the relative pseudocomplement of $a$ relative to $b$, and write $\neg a$ for the pseudocomplement $a\ra\bot$ of $a$.

\begin{lemma}\label{meetirred}
Let $p$ be a meet-irreducible element of $\m A$,  let $a\in \K(\m A)$, and let $b\in A$ such that $a\ra b\leq p$. Then there exists a meet-irreducible element $q$ such that $a\not\leq q$ and $b\leq q\leq p.$ 
\end{lemma}

\begin{proof}
Let $S=\{s\in A\colon s\geq b \text{ and } a\ra s\leq p\}.$ Note that $S\neq\emptyset$ since $b\in S$. By Zorn's Lemma, $S$ has a maximal element $q$. We claim that $q$ is meet-irreducible. Let us assume that this is not the case. Then there exist elements $x, y\in A$ such that $x\not\leq q, y\not\leq q, \text{ and } x\mt y\leq q$. By the maximality of $q$, $a\ra (q\jn x)\not\leq p$, and $a\ra (q\jn y)\not\leq p$. Thus, there exist compact elements $c_1, c_2\not\leq p$ such that  $c_1 \leq a\ra (q\jn x)$ and $c_2 \leq a\ra (q\jn y)$. 
The last two inequalities are equivalent to $a\mt c_1\leq {q\jn x}$ and $a\mt c_2\leq q\jn y$. 
So $a\mt(c_1\mt c_2)\leq (q\jn x)\mt(q\jn y)=q\jn(x\mt y)=q$. Now $c_1\mt c_2\in\K(\m A) \text { and } c_1\mt c_2\not\leq p$. 
Thus, $c_1\mt c_2\leq a\ra q$ and $c_1\mt c_2\not\leq p,$ showing that $a\ra q\not\leq p,$ contradicting the fact that $q\in S$. We have shown that $q$ is meet-irreducible. Lastly, $a\ra q\leq p$ implies that $q\leq p$.
\end{proof}

The minimal elements of the partially ordered set of meet-irreducible elements of $\m A$ will be referred to as \emph{minimal meet-irreducible} elements of $\m A$. 

We state the following result for future reference:

\begin{lemma}\label{minmeetirred2}
Every meet-irreducible element of $\m A$ exceeds a minimal meet-irreducible element. In particular, the meet of all minimal meet-irreducible elements is the bottom element $\bot$.
\end{lemma}

\begin{proof}
A direct application of Zorn's Lemma implies that every meet-irredu\-ci\-ble element exceeds a minimal meet-irreducible element. The second statement is a direct consequence of the fact that in an algebraic lattice every element is a meet of meet-irreducible (in fact, completely meet-irreducible) elements.
\end{proof}

The next result provides a useful characterization of minimal meet-irredu\-ci\-ble elements. 

\begin{lemma}\label{minmeetirred3}
For a meet-irreducible element $p$ of $\m A$, the following statements are equivalent:
\begin{enumerate}
\item[\rm{(1)}] $p$ is a minimal meet-irreducible element;
\item[\rm{(2)}] for all $c\in \K(\m A), c\leq p \text{ implies } \neg c \not\leq p$;
\item[\rm{(3)}] $p=\bigvee\{\neg c\colon c\in \K(\m A), c\not\leq p\}$.
\end{enumerate}
\end{lemma}

\begin{proof} 
(1) implies (2): Assume that (2) fails. Then there exists $a\in \K(\m A)$ such that $a,\neg a \leq p$. 
By Lemma \ref{meetirred}  for $b=\bot$, there exists  a meet-irreducible element $q$ such that $a\not\leq q$ and $b\leq q\leq p$. However, since $a\leq p, $ it follows that $q < p$, and hence $p$ is not a minimal meet-irreducible element of $\m A$.

(2) implies (3): Set $x=\bigvee\{\neg c\colon c\in \K(\m A), c\not\leq p\}.$ It is clear that $x\leq p$. Let now $t\in \K(\m A), t\leq p$. By (2), $\neg t\not\leq p$, and hence there exists compact $s\not\leq p$ such that $t\mt s=\bot$. It follows that $t\leq \neg s\leq x$. This shows that $x=p$.

(3) implies (1): Suppose that (1) fails and (3) holds. Let $q<p$ be a meet-irreducible element. Let $t\in  \K(\m A)$ such that $t\leq p$ and $t\not\leq q$. Now $t\mt \neg t=\bot, $ and hence $\neg t\leq q<p$. Thus, $t, \neg t\leq p$. Since $p=\bigvee\{\neg c\colon c\in \K(\m A), c\not\leq p\},$ there exist compact elements $c_1, \dots, c_n\not\leq p$ such that $t\leq \neg c_1\jn\dots\jn \neg c_n$. So $\neg(\neg c_1\jn\dots\jn \neg c_n)=\neg\neg c_1\mt\dots\mt \neg\neg c_n\leq \neg t\leq p.$  This yields the contradictory inequality $c_1\mt\dots\mt c_n\leq p$.
\end{proof}

Lemmas~\ref{minmeetirred2} and~\ref{minmeetirred3} have the following corollaries in the setting of $\ut$-cyclic residuated lattices.

\begin{proposition}\label{minprime1}
Every prime convex subalgebra of an $\ut$-cyclic residuated lattice $\m L$ exceeds a minimal prime convex subalgebra. In particular, the intersection of all minimal prime convex subalgebras of $\m L$ is $\{\ut\}$.
\end{proposition}

\begin{proposition}\label{minprime2}
For a prime convex subalgebra $P$ of an $\ut$-cyclic residuated lattice $\m L$, the following statements are equivalent:
\begin{enumerate}
\item[\rm{(1)}] $P$ is a minimal prime convex subalgebra of $\m L$;
\item[\rm{(2)}] for all $x\in L, x\in P \text{ implies } x^\perp \not\subseteq P$; and
\item[\rm{(3)}] $P=\bigcup\{x^\perp\colon x\in L\setminus P\}$.
\end{enumerate}
\end{proposition}

\begin{proof}
The proof follows directly from Lemma~\ref{minmeetirred3}. In Condition (3), join can be replaced by union because the set $\{x^\perp\colon x\in L\setminus P\}$ is up-directed.
\end{proof}

Another result that has been proved for various particular classes of residuated lattices is Proposition \ref{polars-primes} below. It is but a direct translation of the following lattice-theoretical result (which does not require that $\K(\m A)$ be a sublattice of $\m A$, see \cite{ST95}).

\begin{lemma}
Let $\m A$ be an algebraic distributive lattice.
For every $a\in A$, 
$\neg a=\bigwedge\{p\in A\colon p \text{ minimal meet-irreducible and } a\nleq p\}$.
\end{lemma}

\begin{proof}
Clearly, if $p$ is a (minimal) meet-irreducible element, then $\neg a\leq p$, since $a\mt\neg a=\bot$.
Let $c$ be a compact element such that $c\nleq\neg a$, i.e. $c\mt a\neq\bot$. By Lemma \ref{minmeetirred2}, there exists a minimal meet-irreducible element $q\in A$ with $c\mt a\nleq q$.
Then $a\nleq q$ and $c\nleq q$, and hence $c\nleq \bigwedge\{p\in A\colon p \text{ minimal meet-irreducible and } a\nleq p\}$.
\end{proof}

\begin{proposition}\label{polars-primes}
Let $\m L$ be an $\ut$-cyclic residuated lattice. For every $H\in\cs{L}$,
$H^\perp=\bigcap\{P\in\cs{L}\colon P \text{ minimal prime and } H\nsubseteq P\}$.
\end{proposition}

In what follows, we generalize another traditional characterization of minimal prime convex $\ell$-subgroups, namely, the one using ultrafilters of the positive cone of an $\ell$-group.
In the setting of $\ut$-cyclic residuated lattices, we prove that a prime convex subalgebra $P$ of $\m L$ is minimal prime iff $L^-\setminus P$ is a maximal ideal of the lattice reduct of $\m L^-$.

\begin{lemma}
Let $\m L$ be an $\ut$-cyclic residuated lattice and let $I\subseteq L^-$. 
Then $I$ is a maximal ideal of $(L^-,\jn,\mt)$ iff $L^-\setminus I$ is a minimal prime\footnote{{Here, a lattice filter or ideal $X$ is \emph{prime} if whenever $x\diamond y\in X$, where $\diamond$ is $\jn$ or $\mt$, then also $x\in X$ or $y\in X$.}} filter of $(L^-,\jn,\mt)$. In this case, $H=\bigcup\{a^\perp\colon a\in I\}$ is the convex subalgebra of $\m L$ generated by $L^-\setminus I$, $H$ is a minimal prime, and $H^-=L^-\setminus I$. 
\end{lemma}

\begin{proof}
We first establish the latter statement, so let $I$ be a maximal ideal of $(L^-,\jn,\mt)$. 
{Recalling that all principal polars $a^\perp$ ($a\in L$) are convex subalgebras of $\m L$ (cf. the remarks after Lemma \ref{polars}), and since the set $\{a^\perp\colon a\in I\}$ is up-directed, we see that $H=\bigcup\{a^\perp\colon a\in I\}$ is a convex subalgebra of $\m L$.}
Furthermore, for all $x\in L^-$ we have: 
\begin{equation}\label{podminkax}
x\in a^\perp \text{ for some } a\in I \quad\text{iff}\quad x\notin I.
\end{equation}
Indeed, if $x\in a^\perp\cap I$ for some $a\in I$, then $\ut=x\vee a\in I$, which contradicts maximality of $I$, and conversely, if $x\notin I$, then the lattice ideal generated by $I\cup\{x\}$ is all of $L^-$, so $x\vee a=\ut$ for some $a\in I$, i.e. $x\in a^\perp$.
It follows that $L^-\setminus I=H^-$, whence we get $H=\C[L^-\setminus I]$.

Condition \eqref{podminkax} also implies that $I$ is a prime ideal.\footnote{A proof here is necessary, since we do not assume that $(L^-,\jn,\mt)$ is a distributive lattice.
{A direct argument: Let $I$ be a maximal ideal of the lattice $(L^-,\jn,\mt)$, and suppose to the contrary that $x\mt y\in I$ for some $x,y\in L^-\setminus I$. Then the ideals generated by $I\cup\{x\}$ and by $I\cup\{y\}$ are both equal to $L^-$, i.e., there exist $a,b\in I$ such that $x\jn a=\ut=y\jn b$. Then, by Lemma \ref{triviallemma}, $\ut=(a\jn b\jn x)(a\jn b\jn y)\leq a\jn b\jn xy\leq a\jn b\jn (x\mt y)\leq \ut$, so $\ut=a\jn b\jn (x\mt y)\in I$, which is a contradiction.}} 
Indeed, if $x,y\in L^-\setminus I$, then there exist $a,b\in I$ such that $x\in a^\perp$ and $y\in b^\perp$. {But we have $a^\perp \cup b^\perp \subseteq (a\vee b)^\perp$ and $(a\vee b)^\perp\in\cs{L}$, hence $x\wedge y\in (a\vee b)^\perp$. Since $a\vee b\in I$, by \eqref{podminkax} we conclude that $x\wedge y\notin I$.}

Now, $L^-\setminus I$ is a prime filter, and it must be minimal prime since for every prime filter $F\subseteq L^-\setminus I$, $L^-\setminus F$ is a prime ideal with $I\subseteq L^-\setminus F$.

Conversely, let $L^-\setminus I$ be a minimal prime filter. Then $I$ is a prime ideal, and if $J$ is a maximal ideal with $I\subseteq J$, then by the first part of the proof we know that $L^-\setminus J$ is a prime filter with $L^-\setminus J \subseteq L^-\setminus I$. Thus $L^-\setminus J = L^-\setminus I$ and $J=I$.
\end{proof}

We have shown that if a lattice with greatest element is the lattice reduct of the negative cone of an \ut-cyclic residuated lattice, then every maximal lattice ideal is a prime ideal; however, the converse is not true. For instance, the pentagon $N_5=\{0,a,b,c,1\}$ -- where $0<1$ are the bounds and $c$ is incomparable to $a<b$ -- cannot be made an integral residuated lattice, though all maximal ideals are prime. Indeed, were it possible, we would have $bc\leq b\mt c=0$ and $a\vee c=1$, whence $b=b(a\jn c)=ba\jn bc=ba\leq a$.

\begin{proposition}
Let $\m L$ be an $\ut$-cyclic residuated lattice.
For every prime convex subalgebra $P\in\cs{L}$, the following are equivalent:
\begin{enumerate}
\item[\rm{(1)}]
$P$ is a minimal prime convex subalgebra of $\m L$;
\item[\rm{(2)}]
$L^-\setminus P$ is a maximal ideal of $(L^-,\jn,\mt)$; and
\item[\rm{(3)}]
$P^-$ is a minimal prime filter of $(L^-,\jn,\mt)$.
\end{enumerate}
\end{proposition}

\begin{proof}
Since $P$ is a prime convex subalgebra, $P^-$ is a prime filter and $I=L^-\setminus P$ is a prime ideal of $(L^-,\jn,\mt)$. Conditions (2) and (3) are equivalent by the previous lemma.

(1) implies (2): 
If the convex subalgebra $P$ is minimal prime, then $P=\bigcup\{a^\perp\colon a\in I\}$.
Hence, if $x\in L^-\setminus I=P^-$, then $x\in a^\perp$ for some $a\in I$. Thus $x\jn a=\ut$ belongs to the lattice ideal generated by $I\cup\{x\}$, proving that $I$ is a maximal ideal.

(2) implies (1): 
Let $x\in P$. Since $\abs{x}\notin I$ and $I$ is a maximal ideal, by \eqref{podminkax} we have $\abs{x}\in a^\perp$ for some $a\in I=L^-\setminus P$. In other words, $a\in x^\perp\setminus P$, i.e. $x^\perp\nsubseteq P$.
By Proposition \ref{minprime2} (2) we conclude that $P$ is a minimal prime convex subalgebra.
\end{proof}


\section{Semilinearity}\label{section<<semilinearity}

Some prominent varieties of residuated lattices -- such as Abelian $\ell$-groups,  MV-algebras, and BL-algebras -- are generated by their linearly ordered members. We refer to such varieties as \emph{semilinear}\footnote{The more traditional, but less descriptive, name for these varieties is \emph{representable}.} varieties, and denote the variety of all semilinear residuated lattices by $\Semrl$.  Thus, a residuated lattice is semilinear iff it is a subdirect product of totally ordered residuated lattices.

Our primary focus in this section is the variety of ($\ut$-cyclic) semilinear residuated lattices. We first recall the axiomatization from \cite{BT03} and \cite{JT02} and then present an alternative axiomatization that does not involve multiplication. The main result of this section is Theorem~\ref{semilinear}, which provides a description of semilinear $\ut$-cyclic residuated lattices in terms of (minimal) prime convex subalgebras and (principal) polars. In particular,  semilinearity can be decided from the lattice of convex subalgebras.
These results extend analogous characterizations of semilinearity in the setting of $\ell$-groups, integral residuated lattices \cite{vA02} and GBL-algebras \cite{Kuhr03b}.

Let $\m L$ be a residuated lattice. Given an element $u \in L$, we define 
\[
\lambda_u(x) = (u\ld xu) \mt \ut \quad\text{and}\quad \rho_u(x)= (ux \rd u) \mt \ut, 
\]
for all $x\in L$.  We refer to $\lambda_u$ and $\rho_u$ as \emph{left conjugation} and \emph{right conjugation} by $u$. An \emph{iterated conjugation} map is a composition $\gamma = \gamma_1  \gamma_2 
\dots  \gamma_n$, where each $\gamma_i$ is a right conjugation or a left conjugation by an element $u_i\in L$. The set of all iterated conjugation maps will be denoted by $\Ga$.

A convex subalgebra $H$ of $\m{L}$ is said to be \emph{normal} if for all $x\in H$ and $y\in L$, $(y\ld xy)\mt\ut\in H$ and $(yx \rd y)\mt\ut\in H$. The following result follows immediately from the definition of a normal convex subalgebra.

\begin{lemma}[\cite{BT03},~\cite{JT02}]\label{normallm1} 
For a convex subalgebra $H$ of a residuated lattice $\m{L}$, the following statements are equivalent:
\begin{enumerate}
\item[\rm{(1)}] $H$ is normal.
\item[\rm{(2)}] $H$ is closed under all iterated conjugation maps.
\item[\rm{(3)}] For all $a, b\in L$, $(a\ld b)\mt\ut\in H$ if and only if $(b\rd a)\mt\ut\in H$.
\end{enumerate}
\end{lemma}

The algebraic closure system of normal convex subalgebras of $\m L$ will be denoted by $\nc{L}$. Given a normal convex subalgebra $H$ of $\m L$, $\Theta_{H}=\{\lan x,y\ran\in L^{2}\colon (x\ld y)\mt (y\ld x)\mt\ut\in H\} $ is a congruence of $\m L$. Conversely, given a congruence $\Theta $, the equivalence class $[\ut]_\Theta$ is a normal convex subalgebra. Moreover:

\begin{lemma}[\cite{BT03},~\cite{JT02}; see also~\cite{Wil06} or~\cite{GJKO07}]\label{idealvar} 
The lattice $\nc{L}$ of normal convex subalgebras of a 
residuated lattice $\m{L}$ is isomorphic to its congruence
lattice $\cng(\m{L})$. The isomorphism is given by the mutually inverse
maps ${H} \mapsto \Theta_{H}$ and $ \Theta \mapsto
[\ut]_\Theta$.
\end{lemma}
 
In what follows, if $H$ is a normal convex subalgebra of $\m L$, we write $\m L\rd H$ for the quotient algebra $\m L\rd \Theta_{H}$, and denote the equivalence class of an element $x\in L$ by $[x]_H$.  We mention that $\nc{L}$ is an algebraic distributive lattice, for any residuated lattice $\m L$. This can be verified directly, or be derived as a consequence of the fact that $\m L$ has a lattice reduct, and hence it is congruence distributive. We also mention the trivial fact that  in a commutative  residuated lattice, every convex subalgebra is normal. 

We will make use of the following auxiliary result:

\begin{lemma}[\cite{BT03},~\cite{JT02}]\label{congruencegeneration}
Let $\m L$ be a residuated lattice, let $S \subseteq L$, and let $\abs{S}$ denote the set of absolute values of elements of $S$.  Let $\Ga$ be the set of all iterated conjugate maps on $\m L$, let  
$\Ga[\abs{S}]=\{\ga(a)\colon a \in \abs{S}, \gamma \in \Ga \}$, and let  $\lan\Ga[\abs{S}]\ran$ be the submonoid of $\m L$ generated by $\Ga[\abs{S}]$. Then: 
\begin{enumerate}
\item[\rm{(1)}] The normal convex  subalgebra $\NC[S]$ of $\m L$ generated by $S$ is 
\begin{align*}
\NC[S]=\NC[\abs{S}] &= \{x\,\in L\colon y\leq x\leq y\ld\ut, \text{ for some } y\in \lan\Ga[\abs{S}]\ran\} \\
&= \{x\,\in L\colon y\leq \abs{x}, \text{ for some } y\in \lan\Ga[\abs{S}]\ran\}.
\end{align*}
\item[\rm{(2)}] The normal convex  subalgebra $\NC[a]$ of $\m L$ generated by an element $a\in L$ is
\begin{align*}
\NC[a]=\NC[\abs{a}] &= \{x\,\in L\colon y\leq x\leq y\ld\ut, \text{ for some } y\in \lan\Ga[\abs{a}]\ran\} \\
&= \{x\,\in L\colon y\leq \abs{x}, \text{ for some } y\in \lan\Ga[\abs{a}]\ran\}.
\end{align*}
\item[\rm{(3)}] $\NC[\abs{a}\jn \abs{b}]\subseteq \NC[a]\cap \NC[b]$ and $\NC[a]\jn \NC[b]=\NC[\abs{a}\mt \abs{b}]$, for all $a,b\in L$.
\end{enumerate}
\end{lemma}

Once we know that the normal convex subalgebra $\NC[S]$ generated by a set $S$ is the convex subalgebra $\C[\Ga[\abs{S}]]$ generated by $\Ga[\abs{S}]$, we can easily prove the following:

\begin{proposition}
For an $\ut$-cyclic residuated lattice $\m L$, the lattice $\nc{L}$ is a complete sublattice of $\cs{L}$.
\end{proposition}

\begin{proof}
Let $\{H_i\colon i\in I\}$ be a family of normal convex subalgebras of $\m L$, and put $H=\bigcup_{i\in I} H_i$.
Let $x\in \bigvee^{\nc{L}}_{i\in I} H_i =\NC[H]=\NC[\abs{H}]$, i.e., $\abs{x}\geq y$ for some $y\in\lan\Ga[\abs{H}]\ran$. But the $H_i$'s are normal, thus $\Ga[\abs{H}]=\bigcup_{i\in I}\Ga[\abs{H_i}]=\bigcup_{i\in I}\abs{H_i}=\abs{H}$. Hence $y\in\lan\abs{H}\ran$ and $x\in \C[H]=\bigvee^{\cs{L}}_{i\in I} H_i$. 
The converse inclusion is obvious, hence the joins in $\nc{L}$ are the same as those in $\cs{L}$.
\end{proof}

The following result, which follows directly from Lemma~\ref{lemma<<primes} and Corollary~\ref{linearity}, will be useful in our considerations.

\begin{lemma}\label{primequotient}
Let $\m L$ be an $\ut$-cyclic residuated lattice that satisfies one of the prelinearity laws. If $H$ is a normal prime convex subalgebra of $\m L$, then $\m{L}\rd H$ is totally ordered.
\end{lemma}

It was shown in \cite{BT03} and \cite{JT02} that $\Semrl$ can be axiomatized -- relative to $\vty{RL}$ -- by either of the equations below:
\begin{gather}
\label{semilinearaxioms1a}
\lambda_u((x\jn y)\ld x)\jn \rho_v((x\jn y)\ld y)\eq \ut,\\
\label{semilinearaxioms1b}
\lambda_u(x\rd(x\jn y))\jn \rho_v(y\rd (x\jn y))\eq \ut.
\end{gather}

Note that the substitution $u=v=\ut$ shows that the left prelinearity law \prp{LP} is an immediate consequence of \eqref{semilinearaxioms1a}, 
and likewise \prp{RP} follows from \eqref{semilinearaxioms1b}. 

Also note that for $\ell$-groups, by substituting $u=y=\ut$ in \eqref{semilinearaxioms1a} we obtain the identity $x \jn vx^{-1}v^{-1} \geq \ut$, which axiomatizes semilinear $\ell$-groups (see e.g. \cite{AF88} or \cite{Dar95}).

The next theorem generalizes the classical results on semilinear $\ell$-groups as well as all analogous results characterizing semilinear members of some classes of residuated lattices -- see \cite{Dvu01} for pseudo-MV-algebras, \cite{Kuhr03c} for pseudo-BL-algebras, \cite{Kuhr03b} for GBL-algebras (DR$\ell$-monoids), and \cite{vA02} for integral residuated lattices.

\begin{theorem}\label{semilinear}
For a variety $\vty{V}$ of residuated lattices, the following statements are equivalent: 
\begin{enumerate}
\item[\rm{(1)}] $\vty{V}$ is semilinear.
\item[\rm{(2)}] $\vty{V}$ satisfies either of the equations \eqref{semilinearaxioms1a} and \eqref{semilinearaxioms1b}.
\item[\rm{(3)}] $\vty{V}$ satisfies either of the prelinearity laws and the quasi-idenitity
\begin{equation}\label{semilinearaxioms2}
x\jn y \ \eq\  \ut \quad\Rightarrow\quad \lam_u(x) \jn \rho_v(y) \ \eq\  \ut.
\end{equation}
\end{enumerate}
If in addition $\vty{V}$ is a variety of $\ut$-cyclic residuated lattices and satisfies either of the prelinearity laws, the preceding conditions are equivalent to each of the conditions below:
\begin{enumerate}
\item[\rm{(4)}]  For every $\m L\in\vty{V}$, all (principal) polars in $\m L$ are normal.
\item[\rm{(5)}] For every $\m L\in\vty{V}$, all minimal prime convex subalgebras of $\m L$ are normal.
\end{enumerate}
\end{theorem}

\begin{proof}
The equivalence of the first three conditions was established in \cite{BT03} and \cite{JT02}. We provide a streamlined proof here for the convenience of the reader.

(1) implies (2): The equations \eqref{semilinearaxioms1a} and \eqref{semilinearaxioms1b} hold in semilinear residuated lattices because they hold in totally ordered ones.

(2) implies (3): Let us assume that 
\eqref{semilinearaxioms1a} holds. As was noted earlier, the substitution $u=v=\ut$ demonstrates that the left prelinearity law \prp{LP} holds. Let now $\m L\in \vty{V}$, and $x, y\in L$ such that $x\jn y=\ut$. Note that $x=\ut\ld x= (x\jn y)\ld x, \text { and likewise, } y=(x\jn y)\ld y.$ Thus, invoking the equation, we get $\lam_u(x)\jn \rho_v(y)=\lam_u((x\jn y)\ld x))\jn \rho_v((x\jn y)\ld y))= \ut$ for all $u, v\in L$. Condition (3) is established.

(3) implies (1): Let $\m L\in\vty{V}$. 
First, we observe that if $x,y\in L$ are such that $x\jn y=\ut$, then the quasi-identity \eqref{semilinearaxioms2} entails that $\lambda_u(x)\jn y=\ut=x\jn\rho_v(y)$ for all $u,v\in L$, whence it easily follows that $\gamma_1(x) \jn \gamma_2(y) = \ut$ for all $\gamma_1,\gamma_2\in\Ga$.
Second, if $x_i,y_j\in L$ are such that $x_i \jn y_j=\ut$ for $i=1,\dots,m$ and $j=1,\dots,n$, then $x_1\dots x_n \jn y_1\dots y_n = \ut$ by Lemma \ref{triviallemma}.
Recalling Lemma \ref{congruencegeneration} (2), these two simple observations imply that $\NC[x] \cap \NC[y]=\{\ut\}$ whenever $x\jn y=\ut$.

Now, suppose that $\m L\in\vty{V}$ is subdirectly irreducible and that $\vty{V}$ satisfies the left prelinearity law \prp{LP}. Then, for any $a,b\in L$, we have $\NC[(a\ld b)\mt\ut] \cap \NC[(b\ld a)\mt\ut] = \{\ut\}$, which -- owing to subdirect irreducibility -- is possible only if  $\NC[(a\ld b)\mt\ut] = \{\ut\}$ or $\NC[(b\ld a)\mt\ut] = \{\ut\}$. In the former case, we get $(a\ld b)\mt\ut=\ut$ and $a\leq b$, while in the latter case, $(b\ld a)\mt\ut=\ut$ and $b\leq a$. Thus $\m L$ is totally ordered.

(1) implies (4): In order to show that every polar is normal, it will suffice to show that every principal polar is normal. To this end, consider a principal polar 
\[
h^\perp=\{a\in L\colon \abs{a}\jn \abs{h}=\ut\}. \quad\text{(Refer to Lemma~\ref{polars}.)}
\]
Let $a\in h^\perp$ and $u\in L$. Then $\abs{a}\jn \abs{h}=\ut$ yields $\ut=\lam_u(\abs{a}) \jn \rho_\ut(\abs{h})=\lam_u(\abs{a}) \jn \abs{h}$, so $\lam_u(\abs{a}) \in h^\perp$. But $\abs{a}\leq a$ implies $\lam_u(\abs{a}) \leq \lam_u(a)$, and hence $\lam_u(a)\in h^\perp.$ Likewise, $\rho_u(a)\in h^\perp.$ We have established that $h^\perp$ is a normal convex subalgebra of $\m L$.

(4) implies (5): Let $H\in\cs{L}$ be a minimal prime convex subalgebra of $\m L$. In light of Proposition ~\ref{minprime2}, $H=\bigcup\{h^\perp\colon h\in L\setminus H\}$. Since each $h^\perp$ is normal, then so is $H$.

(5) implies (1): Let $\{H_i\colon i\in I\}$ be the set of minimal prime convex subalgebras of $\m L$. In view of Lemma \ref{primequotient}, each quotient algebra $\m L/H_i$ is totally ordered. Moreover, by Proposition ~\ref{minprime1}, the intersection of the $H_i$'s is just $\{\ut\}$. It follows that $\m L$ is a subdirect product of the residuated chains $\m L/H_i$ ($i\in I$), and thus it is semilinear. 
\end{proof}

A direct inspection of the proof shows that instead of the equations \eqref{semilinearaxioms1a} and \eqref{semilinearaxioms1b} we could equally use the equations
\begin{gather*}
\lambda_u((y\ld x)\mt\ut)\jn \rho_v((x\ld y)\mt\ut)\eq \ut,\\
\lambda_u((x\rd y)\mt\ut)\jn \rho_v((y\rd x)\mt\ut)\eq \ut.
\end{gather*}

It should be noted that the equations \eqref{semilinearaxioms1a} and \eqref{semilinearaxioms1b}, as well as the above equations, involve all the operation symbols. It is interesting to note that a purely  implicational characterization of the variety $\Semirl$ of semilinear integral residuated lattices, relative to the variety $\Irl$ of integral residuated lattices, was conjectured in~\cite{vA02} and proven in~ \cite{Kur07}. The defining equation is the following:
\[
((x\ld y)\ld u)\ld \big[\big(\big[w\rd (w\rd [((y\ld x)\ld z)\ld z])\big]\ld u\big)\ld u\big] \eq\ut.
\]

Theorem~\ref{newchar} presents a characterization of $\Semrl$ that does  not involve multiplication.

A \emph{$*$-conjugate map} is a map of the form $\lambda^*_u(x)=((x\ld u)\ld u)\mt\ut$ or $\rho^*_u(x)=$ $(u\rd (u\rd x))\mt\ut$. 
We have the following analogue of Lemma \ref{normallm1}:

\begin{lemma}\label{starlamrho}
Let $\m L$ be an arbitrary residuated lattice. A convex subalgebra $H\in\cs{L}$ is normal if and only if  $\lambda^*_u(x) \text{ and } \rho^*_u(x) \in H,$ for all $x\in H$ and $u\in L$.
\end{lemma}

\begin{proof}
Let $H$ be a convex normal subalgebra of $\m L$. Consider $x\in H$ and  $u\in L$. We proceed to show that $\lambda^*_u(x)\in H$ (with the proof of $\rho^*_u(x)\in H$ being analogous). By assumption, $\lam_{(x\ld u)}(x)=[(x\ld u)\ld x(x\ld
u)]\mt \ut\in H$. Thus, taking into account the inequality $x(x\ld u)\leq u$, we get
$\lam_{(x\ld u)}(x)\leq [(x\ld u)\ld u]\mt \ut=\lambda^*_u(x)\leq\ut.$ Thus, 
$\lambda^*_u(x)\in H$.

Conversely, suppose that $\lambda^*_u(x) \text{ and } \rho^*_u(x) \in H,$ for all $x\in H$ and $u\in L$. Let $x\in H$ and  $u\in L$. Then $\lambda^*_{xu}(x)=((x\ld xu)\ld xu)\mt \ut\in H$. Since $u\leq x\ld xu$, we have 
$\lambda^*_{xu}(x) \leq (u\ld xu)\mt \ut=\lam_u(x)\leq \ut$. Thus $\lam_u(x)\in H$. Likewise, $\rho_u(x)\in H$.
\end{proof}

An \emph{$*$-iterated conjugation} map is a composition $\gamma^* = \gamma^*_1  \gamma^*_2 \dots  \gamma^*_n$, where each $\gamma^*_i$ is a $\lambda^*_u$ or a $\rho^*_u$, for element $u\in L$. 
The set of all $*$-iterated conjugation maps will be denoted by $\Ga^*$.

The next result is an analogue of Lemma~\ref{congruencegeneration}. It follows from that result and the proof of Lemma~\ref{starlamrho}.

\begin{lemma}\label{congruencegeneration2}
Let $\m L$ be a residuated lattice, let $S \subseteq L$, and let $\abs{S}$ denote the set of absolute values of elements of $S$.  Let $\Ga^*$ be the set of all $*$-iterated conjugate maps on $\m L$, let  $\Ga^*[\abs{S}]=\{\ga^*(a)\colon a \in \abs{S}, \gamma^* \in \Ga^* \}$, and let  $\lan\Ga^*[\abs{S}]\ran$ be the submonoid of $\m L$ generated by $\Ga^*[\abs{S}]$. Then: 
\begin{enumerate}
\item[\rm{(1)}] The normal convex  subalgebra $\NC[S]$ of $\m L$ generated by $S$ is 
\begin{align*}
\NC[S]=\NC[\abs{S}] &= \{x\,\in L\colon y\leq x\leq y\ld\ut, \text{ for some } y\in \lan\Ga^*[\abs{S}]\ran\} \\
& = \{x\,\in L\colon y\leq \abs{x}, \text{ for some } y\in \lan\Ga^*[\abs{S}]\ran\}.
\end{align*}
\item[\rm{(2)}] The normal convex  subalgebra $\NC[a]$ of $\m L$ generated by an element $a\in L$ is
\begin{align*}
\NC[a]=\NC[\abs{a}]&= \{x\,\in L\colon y\leq x\leq y\ld\ut, \text{ for some } y\in \lan\Ga^*[\abs{a}]\ran\} \\
&= \{x\,\in L\colon y\leq \abs{x}, \text{ for some } y\in \lan\Ga^*[\abs{a}]\ran\}.
\end{align*}
\end{enumerate}
\end{lemma}

The following theorem gives the promised alternative axiomatization of $\Semrl$ in terms of $*$-conjugations $\lambda^*_u$ and $\rho^*_u$.

\begin{theorem}\label{newchar}
For a variety $\vty{V}$ of residuated lattices, the following statements are equivalent: 
\begin{enumerate}
\item[\rm{(1)}] $\vty{V}$ is semilinear;

\item[\rm{(2)}] $\vty{V}$ satisfies either of the equations 
\begin{gather}
\label{semilinearaxioms3a}
\lambda^*_u ((x\jn y)\ld x) \jn \rho^*_v  ((x\jn y)\ld y) \eq \ut,\\
\label{semilinearaxioms3b}
\lambda^*_u (x\rd (x\jn y)) \jn \rho^*_v (y\rd (x\jn y)) \eq \ut;
\end{gather}
\item[\rm{(3)}] $\vty{V}$ satisfies either of the prelinearity laws and the quasi-idenitity
\begin{equation}\label{semilinearaxioms4}
x\jn y \ \eq\  \ut \quad\Rightarrow\quad \lambda^*_u(x) \jn \rho^*_v(y) \ \eq\  \ut.
\end{equation}

\end{enumerate}
\end{theorem}

\begin{proof} 
We can mimic the corresponding part of the proof of Theorem~\ref{semilinear}, or we can simply observe that the (quasi-)identities \eqref{semilinearaxioms1a}, \eqref{semilinearaxioms1b} and \eqref{semilinearaxioms2} are equivalent to the (quasi-)identities \eqref{semilinearaxioms3a}, \eqref{semilinearaxioms3b} and \eqref{semilinearaxioms4}, respectively.
Indeed, when proving Lemma~\ref{starlamrho}, we have actually shown that the inequalities 
$\lambda^*_{xy}(x) \leq \lambda_y(x)$ and $\lambda_{x\ld y}(x) \leq \lambda^*_y(x)$, 
and symmetrically, $\rho^*_{yx}(x) \leq \rho_y(x)$ and $\rho_{y\rd x}(x)\leq\rho^*_y(x)$, 
hold in any residuated lattice $\m L$. 
Hence, it is easy to see that, for any $x,y\in L$, if $\m L$ satisfies $\lambda_u(x) \jn \rho_v(y)=\ut$ for all $u,v\in L$, then $\m L$ satisfies $\lambda^*_u(x) \jn \rho^*_v(y)=\ut$ for all $u,v\in L$, and vice versa.
\end{proof}

We should remark that similarly to Theorem~\ref{semilinear}, the identities \eqref{semilinearaxioms3a} and \eqref{semilinearaxioms3b} in Theorem~\ref{newchar} could be replaced by the identities
\begin{gather*}
\lambda^*_u((y\ld x)\mt\ut)\jn \rho^*_v((x\ld y)\mt\ut)\eq \ut,\\
\lambda^*_u((x\rd y)\mt\ut)\jn \rho^*_v((y\rd x)\mt\ut)\eq \ut.
\end{gather*}


\section{Hamiltonian residuated lattices}\label{section<<hamiltonian}

This section is devoted to $\ut$-cyclic residuated lattices whose convex subalgebras are normal. 
An $\ut$-cyclic residuated lattice with this property will be called \emph{Hamiltonian}.\footnote{The term is borrowed from group theory, where it is usually used to designate a non-commutative group whose subgroups are normal. A typical example of a Hamiltonian group is  the  quaternion group.} 
Clearly, every Abelian $\ell$-group is Hamiltonian, while $\ell$-groups whose group reducts are nilpotent groups serve as interesting examples of non-commutative residuated lattices that are Hamiltonian (\cite{Kop75}).
Moreover, it is known that the class $\mathcal{H}am\mathcal{LG}$ of all $\ell$-groups with this property is merely a torsion class, while the largest variety contained in $\mathcal{H}am\mathcal{LG}$ is the variety of weakly Abelian $\ell$-groups (\cite{Mart72}, \cite{Rei83}).
As we will see, the analogy fails to be true for $\ut$-cyclic residuated lattices. More specifically, there is no maximal variety of Hamiltonian $\ut$-cyclic residuated lattices.

We first characterize individual Hamiltonian residuated lattices:

\begin{lemma}\label{lemma<<Hamiltonian}
For an $\ut$-cyclic residuated lattice $\m L$, the following are equivalent:
\begin{enumerate}
\item[\rm{(1)}]
$\m L$ is Hamiltonian.
\item[\rm{(2)}]
For all $a\in L^-$ and $b\in L$,  there exist $m,n\in\N$ such that $a^m\leq \lam_b(a)$ and $a^n\leq \rho_b(a)$.
\item[\rm{(3)}]
For all $a, b\in L$ there exist $m,n\in\N$ such that $(a\mt \ut)^m\leq \lam_b(a)$ and $(a\mt\ut)^n\leq \rho_b(a)$.
\item[\rm{(4)}]
For all $a,b\in L$ there exist $m,n\in\N$ such that $\abs{a}^m\leq  \lam_b(a)$ and $\abs{a}^n\leq \rho_b(a)$.
\end{enumerate}
\end{lemma}

\begin{proof} 
(1) implies (2):
For every $a\in L^-$, the convex subalgebra $\C[a]$ is normal, so $\lam_b(a), \rho_b(a)\in \C[a]$ for every $b\in L$. Hence, by Corollary~\ref{CS1}, there exist $m,n\in\N$ such that $a^m\leq \lam_b(a)$ and $a^n\leq \rho_b(a)$.

(2) implies (3):
Clearly, $(a\mt\ut)^m\leq \lam_b(a\mt\ut)\leq \lam_b(a)$ and, likewise, $(a\mt\ut)^n\leq \rho_b(a)$.

(3) implies (4): 
It suffices to observe that $\abs{a}^k\leq (a\mt\ut)^k$ for any $k\in\N$.

(4) implies (1):
Let $H\in\cs{L}$. For all $a\in H$ and $b\in L$, there exist $m,n\in\N$ with such that $\abs{a}^m\leq  \lam_b(a)$ and $\abs{a}^n\leq \rho_b(a)$. This means that $\lam_b(a), \rho_b(a)\in\C[a]$, and since $\C[a]\subseteq H$, we conclude that $H$ is normal.
\end{proof}

A class of residuated lattices will be called \emph{Hamiltonian} if all its members are Hamiltonian.

\begin{theorem}\label{thm<<hamiltonian}
Let $\mathcal{K}$ be a class of $\ut$-cyclic residuated lattices that is closed under direct products. Then $\mathcal{K}$ is Hamiltonian if and only if there exist $m,n\in\N$ such that $\mathcal{K}$ satisfies the identities
\begin{equation}\label{EQ:HAM1}
(x\mt\ut)^m\leq \lam_y(x) \quad\text{and}\quad (x\mt\ut)^n\leq \rho_y(x),
\end{equation}
or, equivalently, the identities
\begin{equation}\label{EQ:HAM2}
\abs{x}^m\leq \lam_y(x) \quad\text{and}\quad \abs{x}^n\leq \rho_y(x).
\end{equation}
\end{theorem}

\begin{proof}
Let $\mathcal{K}$ be Hamiltonian and suppose to the contrary that $\mathcal{K}$ does not satisfy the identity $(x\mt\ut)^m\leq  \lam_y(x)$ for any $m\in\N$ (the case when the identity $(x\mt\ut)^n\leq \rho_y(x)$ fails for all $n\in\N$ is parallel).
Then for every $i\in\N$ there exists $\m L_i\in\mathcal{K}$ and $a_i,b_i\in L_i$ such that $(a_i\mt\ut)^i\nleq \lam_{b_i}(a_i)$.
Let $\m L$ be the direct product $\prod_{i\in \N} \m L_i$. By the hypothesis, $\m L$ belongs to $\mathcal{K}$, so it is Hamiltonian. However, if we take the sequences $a=\lan a_1,a_2,\dots\ran$ and $b=\lan b_1,b_2,\dots\ran$, then  $(a\mt\ut)^m\nleq \lam_b(a)$,  for every $m\in\N$. Thus, in view of Lemma \ref{lemma<<Hamiltonian}, $\m L$ cannot be Hamiltonian, since the convex subalgebra $\C[a\mt\ut]$ of $\cs{L}$ is not normal.

The opposite direction is obvious: if $\mathcal{K}$ satisfies \eqref{EQ:HAM1} for some $m,n\in\N$, then by Lemma \ref{lemma<<Hamiltonian}, every $\m L$ in $\mathcal{K}$ is Hamiltonian.

The same arguments show that $\mathcal{K}$ is Hamiltonian exactly if it fulfills \eqref{EQ:HAM2} for some $m,n\in\N$.
\end{proof}

Note that in Lemma \ref{lemma<<Hamiltonian}, as well as in Theorem \ref{thm<<hamiltonian}, the conjugations $\lambda_u,\rho_u$ can be replaced by $\lambda^*_u,\rho^*_u$.

It is known (\cite{Mart72}, \cite{Rei83}) that there exists a largest Hamiltonian variety of $\ell$-groups, viz., the variety of \emph{weakly Abelian} $\ell$-groups, which is defined by the identity $(x\mt\ut)^2 y\leq yx$. This identity bears striking similarity with the identities $y(x\mt\ut)^m\leq xy$ and $(x\mt\ut)^n y\leq yx$, which are equivalent versions of  the identities \eqref{EQ:HAM1} above. The natural question arises whether arbitrary values of $m$ and $n$ in the preceding equations can be replaced by $m=n=2$, or some other fixed value.  The next result shows that this is not the case. In particular, there is no largest Hamiltonian variety of $\ut$-cyclic residuated lattices.

\begin{theorem}\label{thm<<hamiltonian2}
Let $\mathcal{H}_{m,n}$ denote the Hamiltonian variety of $\ut$-cyclic residuated lattices satisfying the identity \eqref{EQ:HAM1}.  
Then the map $\lan m,n\ran \mapsto \mathcal{H}_{m,n}$ is an order-embedding of $(\N^2,\leq)$ into the lattice of Hamiltonian varieties of $\ut$-cyclic residuated lattices.  In particular, there is no largest Hamiltonian variety of $\ut$-cyclic residuated lattices.
\end{theorem}

\begin{proof}
Let $\mathcal{H}_{m,n}$ be the Hamiltonian variety defined by \eqref{EQ:HAM1}. 
It is clear that $\mathcal H_{k,l}\subseteq \mathcal H_{m,n}$ whenever $\lan k,l\ran \leq \lan m,n\ran$ in $(\N^2,\leq)$. 
By the previous theorem, every Hamiltonian variety is contained in $\mathcal{H}_{m,n}$ for some $m,n\in\N$, thus a maximal Hamiltonian variety would be of this form.

First, we construct an example of a Hamiltonian $\ut$-cyclic residuated lattice (in fact, a linearly ordered pseudo-MV-algebra) that belongs to $\mathcal{H}_{n,1}$ but not to $\mathcal{H}_{n-1,1}$ (for an arbitrary fixed $n\in \N$, $n\geq 2$).

Given $n\in\N$, let $H=\{\langle nx,x\rangle\colon x\in\R^-\}$ and let $\m B_n$ be the residuated chain 
$(H\cup\R^+,\pd,\ld,\rd,\jn,\mt,\lan 0,0\ran)$, 
where $H$ and $\R^+$ are linearly ordered in the natural way and $x\leq \langle ny,y\rangle$ for all $x\in\mathbb{R}^+$ and $\langle ny,y\rangle\in H$, and where the operation $\cdot$  is defined by
\begin{align*}
\langle nx,x\rangle \pd \langle ny,y\rangle &=\langle n(x+y),x+y\rangle &&\mbox{if }\langle nx,x\rangle,\langle ny,y\rangle\in H,\\
x \pd \langle ny,y\rangle &= 0\vee (x+y) &&\mbox{if } x\in \R^+, \langle ny,y\rangle\in H,\\
\langle ny,y\rangle \pd x &= 0\vee (ny+x) &&\mbox{if } x\in \R^+, \langle ny,y\rangle\in H,\\
x \pd y&=0 &&\mbox{if } x,y\in\R^+.
\end{align*}
The residuals $\ld$ and $\rd$ are then defined as follows:
\begin{align*}
\langle nx,x\rangle \ld\langle ny,y\rangle &=\langle 0\mt n(y-x),0\mt (y-x)\rangle &&\mbox{if }\langle nx,x\rangle,\langle ny,y\rangle\in H,\\
x \ld \langle ny,y\rangle &= \langle 0,0\rangle &&\mbox{if } x\in \R^+, \langle ny,y\rangle\in H,\\
\langle ny,y\rangle\ld x &=  x-ny &&\mbox{if } x\in \R^+, \langle ny,y\rangle\in H,\\
x\ld y &=\langle 0\mt n(y-x),0\mt (y-x)\rangle &&\mbox{if } x,y\in\R^+,
\end{align*}
\begin{align*}
\langle ny,y\rangle/\langle nx,x\rangle &= \langle 0\mt n(y-x),0\mt (y-x)\rangle &&\mbox{if }\langle nx,x\rangle,\langle ny,y\rangle\in H,\\
\langle ny,y\rangle /x &= \langle 0,0\rangle &&\mbox{if } x\in \R^+, \langle ny,y\rangle\in H,\\
x/\langle ny,y\rangle &=  x-y &&\mbox{if } x\in \R^+, \langle ny,y\rangle\in H,\\
y/x &= \langle 0\mt (y-x),0\mt \textstyle{\frac{1}{n}}(y-x)\rangle &&\mbox{if } x,y\in\R^+.
\end{align*}

Of course, if two elements commute, then they satisfy both inequalities of \eqref{EQ:HAM1}. 
For $\langle nx,x\rangle\in H$ and $y\in\R^+$ we have
\[
y\ld (\langle nx,x\rangle\pd y) \geq \langle n^2x,nx \rangle = \langle nx,x\rangle^n
\]
and
\[
(y\pd \langle nx,x\rangle)/y \geq \langle x,\textstyle{\frac{1}{n}}x \rangle \geq \langle nx,x\rangle.
\]
Hence $\m B_n\in\mathcal{H}_{n,1}$.
However, we also have 
\[
n\ld (\langle -n,-1\rangle \pd n)= \langle -n^2,-n\rangle = \langle -n,-1\rangle^n,
\]
and so $\m B_n\notin\mathcal{H}_{n-1,1}$.  

Symmetrically, for any $n\in\N$ we can construct a Hamiltonian residuated chain, say $\m B^*_n$, that belongs to $\mathcal{H}_{1,n}\setminus\mathcal H_{1,n-1}$.  
Consequently, if we are given $\lan k,l\ran < \lan m,n\ran$, then $\m B_m\times \m B^*_n\in \mathcal{H}_{m,n} \setminus \mathcal{H}_{k,l}$. Thus $\mathcal{H}_{k,l}\subset\mathcal{H}_{m,n}$.
\end{proof}

The construction of the residuated lattice $\m B_n$ in the proof above has been inspired by the ``kite'' residuated lattice from \cite{JM06}, Section 3.

The class of all Hamiltonian $\ell$-groups is not a variety, but it is a torsion class (see e.g. \cite{AF88}, \cite{Dar95}), that is, it is closed under convex $\ell$-subgroups, homomorphic images and joins of convex $\ell$-subgroups.
We can analogously define a class $\vty{K}$ of ($\ut$-cyclic) residuated lattices to be a \emph{torsion class} if $\vty{K}$ is closed under convex subalgebras, homomorphic images and joins of convex subalgebras in the following sense: if $\m L$ is a residuated lattice (not necessarily in $\vty{K}$), then the join in $\cs{L}$ of any family of convex subalgebras of $\m L$ that all belong to $\vty{K}$ again belongs to $\vty{K}$.

Lemma \ref{lemma<<Hamiltonian} implies that the class $\mathcal{H}am\mathcal{RL}$ of all Hamiltonian $\ut$-cyclic residuated lattices is closed with respect to convex subalgebras and homomorphic images. However,  
the simple example below shows that, in contrast to $\ell$-groups, $\mathcal{H}am\mathcal{RL}$ is not closed with respect to joins of convex subalgebras, and hence it is not a torsion class.

\begin{example}[Cf. \cite{vA02}, Example 2]
Let $\m L$ be the integral residuated lattice given as follows: 
\begin{center}
\begin{minipage}{5cm}
\includegraphics[scale=0.2]{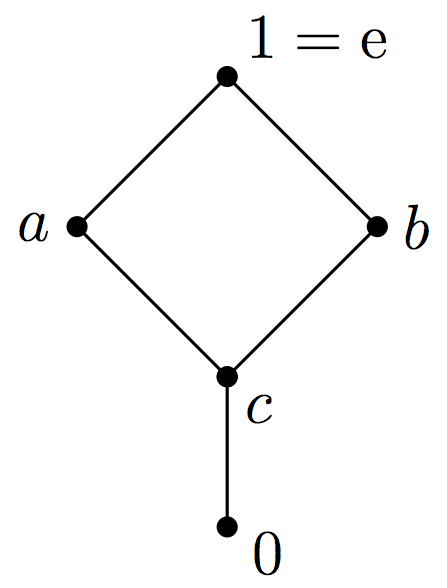}
\end{minipage}
\qquad
\begin{minipage}{4cm}
\begin{tabular}{c|c c c c c c}
  $\cdot$ & $0$ & $a$ & $b$ & $c$ & $\ut$  \\ \hline
         $0$ & $0$ & $0$ & $0$ & $0$ &$0$  \\
         $a$ & $0$ & $a$ & $0$ & $0$ &$a$ \\
         $b$ & $0$ & $c$ & $b$ & $c$ &$b$ \\
         $c$ & $0$ & $c$ & $0$ & $0$ &$c$ \\ 
         $\ut$ & $0$ & $a$ & $b$ & $c$ & $\ut$ 
\end{tabular}
\end{minipage}
\end{center}
Obviously, $\{a,\ut\}$ and $\{b,\ut\}$ are (domains of) Hamiltonian convex subalgebras of $\m L$, but they are not normal in $\m L$ because $\lambda_b(a)=b\ld ab=b\ld 0=0$ and $\rho_a(b)=ab\rd a=0\rd a=0$. Thus $\m L$, which is the join of the two convex subalgebras, itself is not Hamiltonian.
\end{example}

The above example also shows that semilinear $\ut$-cyclic residuated lattices cannot be characterized by the identity $x^2 \jn y^2 \eq (x\jn y)^2$ which characterizes semilinear (representable) $\ell$-groups (e.g. \cite{Gla99}). Indeed, $\m L$ satisfies this identity, though it is not semilinear because the polars $a^\perp=\{b,\ut\}$ and $b^\perp=\{a,\ut\}$ are not normal.

\section{GMV-algebras}\label{section<<gmvalgebras}

The results of the preceding sections demonstrate that lattices of convex subalgebras of $\ut$-cyclic residuated lattices satisfying either prelinearity law bear striking similarities with those of $\ell$-groups. The following question then arises naturally.

\begin{problem}\label{PROBLEM}
Let $\m L$ be an $\ut$-cyclic residuated lattice that satisfies either prelinearity law. Does there exist an $\ell$-group $\m G$ with the property that the lattice $\cs{\m G}$ of its convex $\ell$-subgroups is isomorphic to the lattice $\cs{\m L}$ of convex subalgebras of $\m L$?
\end{problem}  

The main result of this section is Theorem \ref{thm<<gmvalgebras-lgroups}, which asserts that the problem has an affirmative answer when $\m L$ is a GMV-algebra. Recall that a residuated lattice $\m L$ is called
a \emph{GMV-algebra} (generalized MV-algebra) (\cite{GT05}) if $\m L$ satisfies the quasi-identities
$x\leq y$ $\Rightarrow$ $x\rd (y\ld x) \eq y \eq (x\rd y)\ld x$, or equivalently, the identities
$x\rd ((y\ld x)\mt\ut) \eq x\jn y \eq ((x\rd y)\mt\ut)\ld x$.

A larger class is that of GBL-algebras. A \emph{GBL-algebra} (generalized BL-algebra) (\cite{GT05}) if it is \emph{divisible}, i.e., $\m L$ satisfies the quasi-identities 
$x\leq y$ $\Rightarrow$ $(x\rd y) y \eq x \eq y (y\ld x)$, or equivalently, the identities 
$((x\rd y)\mt\ut) y \eq x\mt y \eq y ((y\ld x)\mt\ut)$.

Every GMV-algebra is a GBL-algebra, and all GBL-algebras are $\ut$-cyclic and have distributive lattice reducts. Bounded GMV-algebras are known as \emph{pseudo-MV-algebras}, and the duals of GBL-algebras can be found in the literature under the name \emph{DR$\ell$-monoids} (see, for example, \cite{Kuhr03a}, \cite{Kuhr03b} or \cite{Kuhr05}).

The fundamental theorem about GBL-algebras (\cite{GT05}, Theorem~5.2) states that every GBL-algebra $\m L$ is isomorphic to a direct product of an $\ell$-group and an integral GBL-algebra. 
Specifically, $\m L$ is the direct sum of its subalgebras $\m{G(L)}$ and $\m{I(L)}$, where the domain of the $\ell$-group $\m{G(L)}$ is the set $G(\m L)$ of invertible elements of $\m L$, and the domain of the integral GBL-algebra $\m{I(L)}$ is the set $I(\m L)=\{a\in L \colon a\ld\ut=\ut=\ut\rd a\}$ of \emph{integral} elements of $\m L$.
This result, in the setting of DR$\ell$-monoids, was independently proved by T.~Kov\'{a}\v{r} in his unpublished thesis ``A~general theory of dually residuated lattice ordered monoids'' (Palack\'{y} University, Olomouc, 1996).

If $\m L$ is a GMV-algebra, then its integral part $\m{I(L)}$ is of the form $\m H^-_\gamma$ where $\m H^-$ is the negative cone of an $\ell$-group and $\gamma$ is a \emph{nucleus} on it. That is, $\gamma$ is a closure operator satisfying $\gamma(x)\gamma(y)\leq\gamma(xy)$ for all $x,y\in H^-$. The domain of $\m H^-_\gamma$ is the set $H^-_\gamma=\gamma(H^-)$ and the operations of $\m H^-_\gamma$ are the restrictions to $H^-_\gamma$ of the corresponding operations of $\m H^-$, except that multiplication on $H^-_\gamma$ is defined by $x\circ_\gamma y=\gamma(xy)$. Moreover, $H^-_\gamma$ is a lattice filter in $\m H^-$, and the $\ell$-group $\m H$ can be constructed in such a way that $H^-_\gamma$ generates its negative cone as a monoid. This correspondence extends to a categorical equivalence that generalizes that between MV-algebras and unital Abelian $\ell$-groups \cite{Mun86}, as well as the categorical equivalence between pseudo-MV-algebras and unital $\ell$-groups \cite{Dvu02}. For the details, see \cite{GT05}, Theorem 3.12.

\begin{theorem}\label{thm<<gmvalgebras-lgroups}
The lattice of (normal) convex subalgebras of any GMV-algebra is isomorphic to the lattice of (normal) convex $\ell$-subgroups of some $\ell$-group.
\end{theorem}

\begin{proof}
Let $\m L$ be a GMV-algebra. By \cite{GT05}, $\m L \cong \m G \times \m H^-_\gamma$ where $\m G,\m H$ are $\ell$-groups and $\gamma$ is a nucleus on the negative cone of $\m H$ such that $\lan H^-_\gamma\ran=H^-$.
It is easy to show that the lattice of (normal) convex subalgebras of the direct product of two $\ut$-cyclic residuated lattices is isomorphic to the direct product of their lattices of (normal) convex subalgebras. Hence, it will suffice to prove that the lattice of (normal) convex subalgebras of $\m H^-_\gamma$ is isomorphic to the lattice of (normal) convex $\ell$-subgroups of $\m H$.

Let $\m M = \m H^-_\gamma$. First, let the map $\alpha\colon\cs{\m M}\to\cs{\m H}$ be defined by $\alpha(X)=\C[X]$.
Then $M\cap\C[X]= X$ for every $X\in\cs{M}$. Indeed, if $y\in M\cap\C[X]$, then $\ut\geq y\geq x_1\dots x_n$ for some $x_1,\dots,x_n\in X$. Hence $y=\gamma(y)\geq \gamma(x_1\dots x_n)=x_1 \circ_\gamma \dots \circ_\gamma x_n \in X$, which yields $y\in X$. The converse inclusion is trivial.
Second, let $\beta\colon\cs{\m H}\to\cs{\m M}$ be defined by $\beta(X)=X\cap M$. Then $\beta$ is correctly defined and for every $X\in\cs{H}$ we have $\C[X\cap M]=X$. Indeed, if $x\in X$, then $\abs{x}=a_1\dots a_n$ for some $a_1,\dots,a_n\in M=H^-_\gamma$, since $\lan M\ran=H^-$. We have $\abs{x}\leq a_i\leq\ut$ for each $i$, whence $a_i\in X\cap M$ for each $i$, and so $x\in\C[X\cap M]$. Again, the converse inclusion is trivial.

Since both $\alpha$ and $\beta$ are order-preserving, it follows that they are mutually inverse isomorphisms between the lattices $\cs{M}$ and $\cs{H}$. Also, it is obvious that $\beta(X)\in\nc{M}$ whenever $X\in\nc{H}$.
That $X\in\nc{M}$ implies $\alpha(X)\in\nc{H}$ follows from the fact that, by \cite{GT05}, the category of integral GMV-algebras with homomorphisms is equivalent to the category whose objects are pairs $\lan \m K,\delta\ran$ where $\m K$ is an $\ell$-group and $\delta$ is a nucleus on its negative cone such that $\lan K^-_\delta\ran = K^-$, and whose morphisms 
$\varphi\colon \lan \m K,\delta\ran \to \lan \m K',\delta'\ran$
are homomorphisms $\varphi\colon \m K \to \m K'$
satisfying 
$\varphi{\upharpoonright}_{K^-} \circ \delta = \delta' \circ \varphi {\upharpoonright}_{K^-}$.
Indeed, given $X\in\nc{M}$, let $\lan\m K,\delta\ran$ be the object corresponding to the integral GMV-algebra $\m M/X$, i.e., $\m M/X=\m K^-_\delta$.
Then the natural homomorphism $\nu$ from $\m M$ onto $\m M/X$ can be lifted to a unique morphism $\bar{\nu}\colon\lan\m H,\gamma\ran \to \lan\m K,\delta\ran$. 
If $x\in\ker(\bar{\nu})$, then also $\abs{x}\in\ker(\bar{\nu})$. Since $\lan M\ran=H^-$, there exist $a_1,\dots,a_n\in M$ such that $\abs{x}=a_1\dots a_n$. But $\abs{x}\leq a_i\leq\ut$ yields $a_i\in\ker(\bar{\nu})$, so $\nu(a_i)=[a_i]_X=[\ut]_X=X$ and $a_i\in X$, for each $i$. Thus $x\in\C[X]$.
Conversely, if $x\in\C[X]$, then $\abs{x}\geq y_1\dots y_n$ for some $y_1,\dots,y_n\in X$.
Then $[\ut]_X \geq \bar{\nu}(\abs{x}) \geq \nu(y_1)\dots\nu(y_n)=[\ut]_X$, showing $\bar{\nu}(\abs{x})=[\ut]_X$ and $x\in\ker(\bar{\nu})$.
Thus $\ker(\bar{\nu})$ is just $\alpha(X)=\C[X]$, hence $\alpha(X)\in\nc{H}$.
\end{proof}

There are two other cases for which Problem \ref{PROBLEM} can be answered in the affirmative.  The first is an instance of a lattice-theoretic result in \cite{Mart73}.
As in Section \ref{polarsandprimes}, let $\m A$ be an algebraic distributive lattice with bounds $\bot$ and $\top$, and let $\K(\m A)$ denote the set of compact elements of $\m A$. For any $a\in A$,  the pseudocomplement of $a$ is denoted by $\neg a$.
It is shown in \cite{Mart73} that  if $\m A$ satisfies the equivalent properties (i) and (ii) below, then it is isomorphic to the lattice of $\ell$-ideals (= normal convex $\ell$-subgroups) of some hyper-Archimedean\footnote{An $\ell$-group is called \emph{hyper-Archimedean} if all its homomorphic images are Archimedean. In particular, such an $\ell$-group is Abelian.} $\ell$-group:
\begin{enumerate}
\item[(i)] $c\jn\neg c=\top$ for all $c\in\K(\m A)$;
\item[(ii)] $\K(\m A)$ is a sublattice of $\m A$ and the meet-irreducible elements $p\in A\setminus\{\top\}$ form an antichain. 
\end{enumerate}

Now, by Theorem \ref{CS8} and Lemma \ref{polars}, and since $a^\perp=\C[a]^\perp$ for any $a$, we get:

\begin{theorem}
Let $\m L$ be an \ut-cyclic residuated lattice satisfying either prelinearity law. If $\C[a]\jn a^\perp = L$ for every $a\in L$, or equivalently, if the proper prime convex subalgebras of $\m L$ form an antichain with respect to set-inclusion, then $\cs L$ is isomorphic to the lattice of $\ell$-ideals of some (hyper-Archimedean) $\ell$-group.
\end{theorem}

A second situation that provides an affirmative answer to Problem \ref{PROBLEM} is based on another lattice-theoretic result concerning finite-valued algebraic distributive lattices. Retaining the notation of the preceding two paragraphs, a \emph{value} for $c\in\K(\m A)$  is  a completely meet-irreducible element of $\m A$ that is maximal with respect to not exceeding $c$. 
$\m A$ is called \emph{finite-valued} if each $c\in\K(\m A)$ has finitely many values. It is proved in \cite{Con65}, and exemplified in \cite{HT94}, that any finite-valued algebraic distributive lattice whose compact elements form a relatively normal lattice (refer to Section \ref{polarsandprimes}) is isomorphic to the lattice of $\ell$-ideals of some Abelian $\ell$-group. The last result of this section is an immediate consequence of these facts and Lemma \ref{lemma<<primes} or Proposition \ref{relatnormal}.

\begin{theorem}
Let $\m L$ be an \ut-cyclic residuated lattice satisfying either prelinearity law. If the lattice $\cs{L}$ of convex subalgebras of $\m L$ is finite-valued, then $\cs L$ is isomorphic to the lattice of $\ell$-ideals of some Abelian $\ell$-group.
\end{theorem}

\section{Concluding remarks}\label{section<<concludingremarks}

The preceding sections have provided ample evidence of the key roles $\ell$-groups and the concept of normality play in the study of $\ut$-cyclic residuated lattices. Future developments require a more detailed study of the lattice of convex subalgebras of such algebras. For example, an ongoing project by three of the authors of this article (M. Botur, J. K\"{u}hr, and C. Tsinakis), is devoted to the study of normal-valued $\ut$-cyclic residuated lattices. The term \emph{normal-valued}, borrowed from the theory of $\ell$-groups, refers to a residuated lattice whose completely meet-irreducible convex subalgebras are normal in their cover (in the lattice of convex subalgebras).
One of the early results of the project is that the class of all normal-valued $\ut$-cyclic residuated lattices that satisfy  one of the prelinearity laws is a variety. This generalizes the corresponding result for $\ell$-groups (\cite{Wol68}).

Another direction of research, in which normality plays a role, is connected with the  concept of a \emph{state.} A number of authors have tried to extend the standard concepts of Rie\v{c}an  and Bosbach states for arbitrary bounded residuated lattices.  Since completely simple\footnote{We use the term \emph{completely simple} for a residuated lattice whose only proper convex subalgebra is $\{\ut\}$.} normal-valued $\ut$-cyclic residuated lattices can be of arbitrarily large cardinality, it is unlikely that this line of research will produce useful algebraic information. 
A promising approach would be to introduce an alternative ``ranking function'' that would ``measure'' the relative position of the elements of a completely simple algebra. Further, the time is ripe to study states for non-bounded residuated lattices, in particular cancellative residuated lattices.

\section*{Acknowledgments}

M.B. and J.K. have been supported by the bilateral project ``New perspectives on the residuated posets'' of the Austrian Science Fund (FWF): project I 1923-N25, the Czech Science Foundation (GA\v{C}R): project 15-34697L, and by the Palack\'{y} University projects IGA PrF 2014016 and IGA PrF 2015010 ``Mathematical structures.''

\end{document}